\documentclass[
10pt
]{article}

\usepackage[%
left=1in,
right=1in,
bottom=1in,
top=1in,
footskip=0.3in,
]{geometry}

\usepackage{color}
\definecolor{LinkColor}{rgb}{0,0,1}
\definecolor{LinkColor2}{rgb}{1,0,0}
\definecolor{lbcolor}{rgb}{0.85,0.85,0.85}
\definecolor{FrameColor}{rgb}{0.85,0.85,0.85}

\usepackage[ngerman, english]{babel}
\usepackage[utf8]{inputenc}
\usepackage{enumerate}
\usepackage[normalem]{ulem}
\usepackage{setspace}
\usepackage{multicol}
\usepackage[babel,french=guillemets,german=swiss]{csquotes}

\usepackage{array}
\usepackage{booktabs}
\usepackage{framed}
\usepackage{rotating}
\usepackage{longtable}
\usepackage{multirow}
\usepackage{tabularx}
\newcolumntype{L}[1]{>{\raggedright\arraybackslash}p{#1}} 
\newcolumntype{C}[1]{>{\centering\arraybackslash}p{#1}} 
\newcolumntype{R}[1]{>{\raggedleft\arraybackslash}p{#1}} 

\usepackage{graphicx}
\usepackage{caption}
\usepackage{placeins}

\usepackage{amsmath}
\usepackage{amssymb}
\usepackage{dsfont}
\usepackage{nicefrac}
\usepackage{empheq}
\allowdisplaybreaks
\numberwithin{equation}{section}

\usepackage[%
pdftitle={Titel},%
pdfauthor={Autor},%
pdfcreator={LaTeX, LaTeX with hyperref and KOMA-Script},
pdfsubject={Betreff}, 
pdfkeywords={Keywords}
]{hyperref} 

\hypersetup{%
	colorlinks=true,
	linkcolor=LinkColor,%
	anchorcolor=LinkColor,%
	citecolor=LinkColor2,%
	filecolor=LinkColor,%
	menucolor=LinkColor,%
	pagecolor=LinkColor,%
	urlcolor=LinkColor,%
}

\usepackage[savemem]{listings}
\usepackage{paralist}
{\begin{list}{$\diamondsuit$}{}}%
	{\end{list}}

\usepackage{lineno}

\newcommand*\patchAmsMathEnvironmentForLineno[1]{%
	\expandafter\let\csname old#1\expandafter\endcsname\csname #1\endcsname
	\expandafter\let\csname oldend#1\expandafter\endcsname\csname end#1\endcsname
	\renewenvironment{#1}%
	{\linenomath\csname old#1\endcsname}%
	{\csname oldend#1\endcsname\endlinenomath}}%
\newcommand*\patchBothAmsMathEnvironmentsForLineno[1]{%
	\patchAmsMathEnvironmentForLineno{#1}%
	\patchAmsMathEnvironmentForLineno{#1*}}%
\AtBeginDocument{%
	\patchBothAmsMathEnvironmentsForLineno{equation}%
	\patchBothAmsMathEnvironmentsForLineno{align}%
	\patchBothAmsMathEnvironmentsForLineno{flalign}%
	\patchBothAmsMathEnvironmentsForLineno{alignat}%
	\patchBothAmsMathEnvironmentsForLineno{gather}%
	\patchBothAmsMathEnvironmentsForLineno{multline}%
}

\usepackage{amsthm}
\usepackage{upgreek}

\newtheoremstyle{tstyle}
{15pt}	
{5pt}	
{\itshape}	
{}	
{\bfseries}	
{.}	
{0.5em}	
{}	

\theoremstyle{tstyle}

\newtheorem{theorem}{Theorem}[section]
\newtheorem{lemma}[theorem]{Lemma}

\newtheorem{definition}[theorem]{Definition}
\newtheorem{remark}[theorem]{Remark}

\newtheoremstyle{cstyle}
{15pt}	
{5pt}	
{}	
{}	
{\bfseries}	
{}	
{0.2222em}	
{}	

\theoremstyle{cstyle}

\makeatletter
\g@addto@macro{\thm@space@setup}{\thm@headpunct{}}
\renewenvironment{proof}[1][\proofname]{\par
	\pushQED{\qed}%
	\normalfont \topsep6\p@\@plus6\p@\relax
	\trivlist
	\item[\hskip\labelsep
	\bfseries
	#1\@addpunct{\,}]\ignorespaces
}{%
	\popQED\endtrivlist\@endpefalse
}
\makeatother

\makeatletter
\g@addto@macro{\thm@space@setup}{\thm@headpunct{}}
\newenvironment{sketch-proof}[1][Sketch of the proof]{\par
	\pushQED{\qed}%
	\normalfont \topsep6\p@\@plus6\p@\relax
	\trivlist
	\item[\hskip\labelsep
	\bfseries
	#1\@addpunct{\,}]\ignorespaces
}{%
	\popQED\endtrivlist\@endpefalse
}
\makeatother

\def\R{\mathbb R}
\def\N{\mathbb N}
\def\n{\mathbf n}
\def\V{{\mathcal V^\kappa}}
\def\Vm{{\mathcal V^\kappa_m}}
\def\Vo{{\mathcal V^\kappa_0}}
\def\Va{{\mathcal V^{\kappa\ast}_0}}

\def\eps{\varepsilon}

\def\ddt{\frac{\mathrm d}{\mathrm dt}}

\def\dx{\;\mathrm dx}

\def\dt{\;\mathrm dt}
\def\ds{\;\mathrm ds}
\def\dxt{\;\mathrm d(x,t)}

\def\dS{\;\mathrm dS}
\def\dSx{\;\mathrm dS(x)}

\def\del{\partial}

\def\delt{\partial_{t}}

\def\deln{\partial_\n}

\def\scdot{{\hspace{1pt}\cdot\hspace{1pt}}}

\def\grad{\nabla}
\def\gradg{\nabla_\Gamma}
\def\laplace{\Delta}

\def\nmo{(-\laplace)^{-1}}
\def\nmg{(-\laplace_\Gamma)^{-1}}

\def\tand{\quad\text{and}\quad}
\def\twith{\quad\text{with}\quad}

\def\wto{\rightharpoonup}
\def\wsto{\overset{*}{\rightharpoonup}}

\def\itemi{\item[\textnormal{(i)}]}
\def\itemii{\item[\textnormal{(ii)}]}
\def\itemiii{\item[\textnormal{(iii)}]}

\newcommand{\Underset}[3][0pt]{\ensuremath{\underset{\raise#1\hbox{\small\ensuremath{#2}}}{#3}}}
\newcommand{\Overset}[3][0pt]{\ensuremath{\overset{\raise#1\hbox{\small\ensuremath{#2}}}{#3}}}

\newcommand\numbereq{%
	\ifmeasuring@\else
	\refstepcounter{equation}%
	\fi
	\tag{\theequation}%
}

\begin{document}
	
%
%
	
\begin{center}	
	\LARGE{Weak solutions of the Cahn--Hilliard system with dynamic boundary conditions: A gradient flow approach}
\end{center}

\bigskip

\begin{center}	
	\normalsize{Harald Garcke and Patrik Knopf}\\[1mm]
	\textit{Fakult\"at f\"ur Mathematik, Universit\"at Regensburg, 93053 Regensburg, Germany}\\[1mm]
	\texttt{\href{mailto:Harald.Garcke@ur.de}{Harald.Garcke@ur.de}, \href{mailto:Patrik.Knopf@ur.de}{Patrik.Knopf@ur.de}}
\end{center}
	
%

\bigskip

\begin{center}	
	\textit{This is a preprint version of the paper. Please cite as:}\\
	\textit{H.~Garcke and P.~Knopf, SIAM J. Math. Anal. 52-1 (2020), pp. 340-369.}\\
	\url{https://doi.org/10.1137/19M1258840}
\end{center}

\bigskip

\begin{abstract}
  The Cahn--Hilliard equation is the most common model to
  describe phase separation processes of a mixture of two
  components. For a better description of short-range interactions of
  the material with the solid wall, various dynamic boundary
  conditions have been considered in recent times. New models with
  dynamic boundary conditions have been proposed recently by C.\,Liu
  and H.\,Wu \cite{liu-wu}.  We prove the existence of weak solutions
  to these new models by interpreting the problem as a suitable
  gradient flow of a total free energy which contains volume as well
  as surface contributions. The formulation involves an inner product
  which couples bulk and surface quantities in an appropriate way.  We
  use an implicit time discretization and show that the obtained
  approximate solutions converge to a weak solution of the
  Cahn--Hilliard system. This allows us to
  substantially improve earlier results which needed strong
  assumptions on the geometry of the domain. Furthermore, we prove
  that this weak solution is unique. \\[1ex]
\end{abstract}

\bigskip

%
%

\section{Introduction}
The Cahn--Hilliard equation was originally derived to model spinodal
decomposition in binary alloys. Later it was noticed that the
Cahn--Hilliard equations also describes later stages of the evolution
of phase transition phenomena like Ostwald ripening. Here a particular
important aspect of the Cahn--Hilliard equation is that topological
changes in the interface can be handled directly. This is in contrast
to a classical free boundary description where singularities occur in
cases where the topology changes. Later new application of the
Cahn--Hilliard model appeared in the literature in such areas as
imaging sciences, non-isothermal phase transitions and two-phase
flow. In certain applications it also turned out to be essential to
model boundary effects more accurately. In order to do so several
dynamic boundary conditions have been proposed in the literature which
we are going to review in the following. In this paper we will focus
on a new dynamic boundary condition proposed recently by \cite{liu-wu}
which is in particular important for hydrodynamic applications. It is
expected that these boundary conditions will have an important impact
on the correct modeling of contact line phenomena. Let us now describe
the models studied in the literature in more detail.

Let $T$ and $\eps$ be positive real numbers and let $\Omega\subset\R^d$ (where $d\in\{2,3\}$) be a bounded domain with boundary $\Gamma:=\del\Omega$. The unit outer normal vector on $\Gamma$ will be denoted by $\n=\n(x)$. Then the standard \textit{Cahn--Hilliard equation} (cf. \cite{cahn-hilliard}) reads as follows:
\begin{subequations}
	\begin{align}
	\label{PDE:CH1}
	\delt\phi &= \laplace \mu, \\
	\label{PDE:CH2}
	\mu &= -\eps \laplace \phi + \frac 1 \eps F'(\phi).
	\end{align}
	Usually, this equation is endowed with an initial condition
	\begin{align}
	\label{IC:CH}
	\phi\vert_{t=0}&=\phi_0.
	\end{align}
\end{subequations}

\pagebreak[2]

\noindent Here, $\phi=\phi(x,t)$ and $\mu=\mu(x,t)$ are functions that depend on
time $t\ge 0$ and position $x\in\Omega$. The symbol $\delt$ denotes
the partial derivative with respect to time and $\Delta$ is the
Laplace operator that acts on the variables $x\in\Omega$. The phase
field $\phi$ represents the difference $\phi_1-\phi_2$ of two local
relative concentrations $\phi_i$, $i=1,2$. This means that the regions $\{x\in\Omega: \phi(x)  = 1\}$ and $\{x\in\Omega: \phi(x)  = -1\}$
correspond to the pure phases of the materials while
$\{x\in\Omega: -1<  \phi(x)  <1\}$ represents the diffuse interface between them. The thickness of this interface is proportional to the parameter $\eps>0$. Therefore, $\eps$ is usually chosen very small. The chemical potential $\mu$ can be understood as the Fr\'echet derivative of
the bulk free energy
\begin{align}
E_\text{bulk}(\phi) = \int\limits_\Omega \frac \eps 2|\grad\phi|^2 + \frac 1 \eps F(\phi) \dx.
\end{align}
The function $F$ represents the bulk potential which typically is of
double well form. A common choice is the smooth \textit{double well potential}
\begin{align}
\label{POT:DW}
W_\text{dw}(\phi)=\theta(\phi^2-1)^2, \qquad \theta>0
\end{align} 
which has two minima at $\phi=\pm 1$ and a local (unstable) maximum at $\phi=0$.
The time-evolution of $\phi$ is considered in a bounded
domain and hence suitable boundary conditions have to be
imposed. The homogeneous Neumann conditions
\begin{alignat}{2}
\label{COND:NMU}
\del_\n \mu &= 0 &&\quad\text{on}\;\Gamma\times]0,T[,\\
\label{COND:NPHI}
\del_\n \phi &= 0 &&\quad\text{on}\;\Gamma\times]0,T[
\end{alignat}
are the classical choice. From the no-flux condition \eqref{COND:NMU} we can deduce mass conservation in the bulk
\begin{align}
\int\limits_\Omega \phi(t) \dx = \int\limits_\Omega \phi(0) \dx, \quad t\in[0,T]
\end{align}
and conditions \eqref{COND:NMU},\eqref{COND:NPHI} imply that the bulk free energy is decreasing in the following way:
\begin{align}
\ddt E_\text{bulk}\big(\phi(t)\big) + \int\limits_\Omega |\grad\mu(t)|^2 \dx = 0, \quad t\in]0,T[.
\end{align}
The Cahn--Hilliard equation \eqref{PDE:CH1},\eqref{PDE:CH2} with the
initial condition \eqref{IC:CH} and the homogeneous Neumann conditions
\eqref{COND:NMU} and \eqref{COND:NPHI} is already very well
understood. It has been investigated from many different viewpoints,
for instance see
\cite{abels-wilke,bates-fife,cherfils,
	elliotgarcke,elliot,goldstein,grinfeld,pego,rybka,zheng}.\\[1ex]
However, especially for certain materials, the ansatz
\eqref{COND:NPHI} is not satisfactory as it neglects certain
additional influences of
the boundary to the dynamics in the bulk. For a better description of
interactions between the wall and the mixture, physicists introduced a
surface free energy functional given by
\begin{align}
E_\text{surf}(\phi) = \int\limits_\Gamma \frac {\kappa\eps} 2 |\gradg \phi|^2 + \frac 1 \eps G(\phi) \dS
\end{align}
where $\gradg$ denotes the surface gradient operator on $\Gamma$, $G$
is a surface potential and $\kappa$ is a nonnegative parameter that is
related to interfacial effects at the boundary. In the case $\kappa = 0$ the problem
is related to the moving contact line problem that is studied in
\cite{thompson-robbins}. Moreover, several dynamic boundary conditions
have been proposed in the literature to replace the homogeneous Neumann
conditions (see, e.g.,
\cite{CFP,colli-fukao-ch,colli-sprekels,liero,mininni,miranville-zelik,motoda,racke-zheng,WZ}). Results
for the Allen--Cahn equation (which is another phase-field equation,
cf. \cite{allen-cahn}) with dynamic boundary conditions can be found,
for instance, in \cite{colli-fukao,colli-fukao-lam,
	colli-gilardi}.\\[1ex]
We give some concrete examples of dynamic boundary conditions for the Cahn--Hilliard equation:
\begin{itemize}
	\item The boundary condition
	\begin{align}
	\label{BC:AC}
	\eps\delt\phi = \kappa\eps \laplace_\Gamma \phi - \eps\deln\phi - \frac 1 \eps G'(\phi) \quad\text{on}\;\Gamma\times]0,T[
	\end{align}
	was suggested (e.g. in \cite{kenzler}) to replace the
	condition \eqref{COND:NPHI} and analyzed for example in
	\cite{CGM}. The symbol $\laplace_\Gamma$ denotes the
	Laplace-Beltrami operator on $\Gamma$. Equation \eqref{BC:AC}
	is a surface Allen--Cahn equation which has $\deln\phi$ as an
	additional source term.
	\item The coupled boundary condition
	\begin{align}
	\label{BC:WENZ}
	\left\{
	\begin{aligned}
	\delt\phi &= \sigma\laplace_\Gamma \mu - \deln\mu  &&\text{on}\;\Gamma\times]0,T[\,,\\
	\mu &= -\kappa\eps \laplace_\Gamma \phi + \eps\deln\phi + \frac 1 \eps G'(\phi) &&\text{on}\;\Gamma\times]0,T[
	\end{aligned}
	\right.
	\end{align}
	for some $\sigma\ge 0$ was proposed in \cite{Gal1,goldstein} to replace both \eqref{COND:NMU} and \eqref{COND:NPHI}. In the case $\sigma=0$ (which was first proposed in \cite{Gal1}), this type of dynamic boundary condition is called a \textit{Wentzell boundary condition}. Furthermore, the first line of \eqref{BC:WENZ} directly implies that the total (i.e., bulk plus boundary) mass is conserved.
	\item Another very generic boundary condition to replace \eqref{COND:NPHI} is
	\begin{align}
	\label{BC:EV}
	\left\{
	\begin{aligned}
	\delt\phi &=\laplace_\Gamma \mu_\Gamma  &&\text{on}\;\Gamma\times]0,T[\,,\\
	\mu_\Gamma &= -\kappa\eps \laplace_\Gamma \phi + \eps\deln\phi + \frac 1 \eps G'(\phi) &&\text{on}\;\Gamma\times]0,T[.
	\end{aligned}
	\right.
	\end{align}
	It was derived in \cite{liu-wu} by an energetic variational approach which is based on the following physical principles: Separate conservation of mass both in the bulk and on the surface, dissipation of the total free energy (that is $E:= E_\text{bulk} + E_\text{surf}$) and the force balance both in the bulk and on the boundary. 
	
	Similar to the situation in the bulk, the chemical potential on the boundary $\mu_\Gamma$ can be described by the Fr\'echet derivative of the surface free energy $E_\text{surf}$. However, the interaction term $\eps \del_\n \phi$ must be added. We point out that $\mu_\Gamma$ does not necessarily coincide with the trace of $\mu$ on $\Gamma$ but is to be understood as an independent variable.
\end{itemize}
A mathematical correspondence between the above dynamic boundary conditions and the total free energy
\begin{align}
E(\phi) = E_\text{bulk}(\phi) + E_\text{surf}(\phi)
\end{align}
will be explained in Section 3.\\[1ex]
In this paper, we investigate the Cahn--Hilliard system subject to
the dynamic boundary condition \eqref{BC:EV}. This means that the overall system reads as follows: 
\begin{subequations}
	\label{CHLW:EPS}
	\begin{align}
	&\delt\phi = \laplace \mu \quad\twith \mu = -\eps \laplace \phi + \frac 1 \eps F'(\phi)
	&& \text{in} \;\; \Omega\times]0,T[, \\	
	&\del_\n \mu = 0 
	&& \text{on} \;\; \Gamma\times]0,T[, \\	
	&\delt\phi = \laplace_\Gamma \mu_\Gamma
	\twith \mu_\Gamma = -\kappa\eps \laplace_\Gamma \phi + \eps\del_\n \phi + \frac 1 \eps G'(\phi)
	&& \text{on} \;\; \Gamma\times]0,T[, \\	
	&\phi(\cdot,0) = \phi_0
	&& \text{in} \;\; \Omega.
	\end{align}
\end{subequations}
If a solution of \eqref{CHLW:EPS} is sufficiently regular, it directly follows that the mass in the bulk and the mass on the surface are conserved:
\begin{align}
\int\limits_\Omega \phi(t) \dx = \int\limits_\Omega \phi(0) \dx \tand
\int\limits_\Gamma \phi(t) \dS = \int\limits_\Gamma \phi(0) \dS, \quad t\in[0,T].
\end{align}
Moreover, the total free energy is still decreasing in time in the following sense:
\begin{align}
\ddt E\big(\phi(t)\big) + \int\limits_\Omega |\grad\mu(t)|^2 \dx 
+ \int\limits_\Gamma |\gradg \mu_\Gamma(t)|^2 \dx = 0, \quad t\in]0,T[.
\end{align}
Existence and uniqueness of weak and strong solutions of the system \eqref{CHLW:EPS} have been established by C. Liu and H. Wu in \cite{liu-wu}. The idea of their proof is to construct solutions of a regularized system where the equations for $\mu$ and $\mu_\Gamma$ are replaced by
\begin{align*}
\mu = -\eps \laplace \phi +\alpha \delt\phi + \frac 1 \eps F'(\phi) \tand \mu_\Gamma = -\kappa\eps \laplace_\Gamma \phi + \alpha\delt\phi + \eps\del_\n \phi + \frac 1 \eps G'(\phi).
\end{align*}
Then a solution of the original problem can be found by taking the limit $\alpha\to 0$. However, in the case $\kappa=0$, their proof requires very strong assumptions on the domain $\Omega$ and its boundary $\Gamma$. To be precise, they need the requirement
\begin{align}
\label{COND:DOM}
c_\mathcal{R}|\Gamma|^\frac{1}{2}|\Omega|^{-1} < 1
\end{align}
where $c_\mathcal{R}$ is a constant that results from the inverse trace theorem.\\[1ex]
In this paper, we use a different ansatz to construct weak solutions of the system \eqref{CHLW:EPS}. In our approach the condition \eqref{COND:DOM} is not necessary which is a substantial improvement. However, on the other hand, our notion of a weak solution (cf. Definition \ref{DEF:WS}) is slightly weaker than that in \cite{liu-wu} so both results have their merits. We will show in Section 3 that \eqref{CHLW:EPS} can be regarded as a gradient flow equation of the total free energy, namely
\begin{align}
\langle \delt\phi ,\eta \rangle = -\frac{\delta E}{\delta \phi}(\phi)[\eta] 	
\end{align}
for all admissible test functions $\eta$ where $\langle\cdot,\cdot\rangle$ is a suitable inner product on a certain function space. This representation can be used to construct a weak solution of the system \eqref{CHLW:EPS} by implicit time discretization which will be done in Section 4. In Section 5, we will establish a uniqueness result for the obtained weak solution. Finally, in Section 6, we present several plots of two numerical simulations. In particular, they demonstrate the influence of mass conservation on the boundary. For simplicity and since it does not play any role in the analysis, we will set $\eps=1$ in Sections~3-5.

%
%

\section{Preliminaries}

In this section we introduce some preliminaries that we will use in the rest of this paper:

\begin{enumerate}
	
	\item[\textbf{(P1)}] In the existence and uniqueness results (see Sections 4 and 5), $\Omega\subset\R^d$ (with
	$d\in\{2,3\}$) denotes a bounded domain with Lipschitz boundary
	$\Gamma:=\del\Omega$. In Sections 1-3, however, we assume additionally that $\Omega$ has a $C^2$-boundary. $\n$ denotes the unit outer normal
	vector on $\Gamma$ and $\del_\n$ is the normal derivative on
	$\Gamma$. We denote by $|\Omega|$ the $d$-dimensional Lebesgue
	measure of $\Omega$ and by $|\Gamma|$ the $(d-1)$-dimensional surface
	measure of $\Gamma$. For any $s>0$ we will also write
	$\Omega_s:=\Omega\times]0,s[$ and $\Gamma_s:=\Gamma\times]0,s[$.
	
	\item[\textbf{(P2)}] We assume that the potentials $F$ and $G$ are bounded from below by
	\begin{align*}
	F(\varphi) \ge C_F \tand\quad G(\varphi) \ge C_G, \qquad\text{for all}\; \varphi\in\R
	\end{align*}	
	with constants $C_F,C_G\in\R$. Moreover, we assume that $F$ and $G$ can be written as
	\begin{align*}
	F(\varphi)=F_1(\varphi) + F_2(\varphi) \tand\quad G(\varphi)=G_1(\varphi) + G_2(\varphi)
	, \qquad\text{for all}\; \varphi\in\R
	\end{align*}
	where 
	\begin{enumerate}
		\item[\textbf{(P2.1)}] $F_1,F_2,G_1,G_2\in C^1(\R)$,
		\item[\textbf{(P2.2)}] $F_1$ and $G_1$ are convex and nonnegative,
		\item[\textbf{(P2.3)}] For any $\delta>0$ there exists constants $A_F^\delta,A_G^\delta \ge 0$ such that
		\begin{align*}
		|F_1'(\varphi)|\le \delta F_1(\varphi) + A_F^\delta \tand |G_1'(\varphi)|\le \delta G_1(\varphi) + A_G^\delta
		\end{align*}
		for all $\varphi\in\R$.
		\item[\textbf{(P2.4)}] 
		There exist constants $B_F,B_G \ge 0$ such that 
		\begin{align*}
		|F_2'(\varphi)|\le B_F\big(|\varphi|+1\big) \tand |G_2'(\varphi)|\le B_G\big(|\varphi|+1\big)
		\end{align*}
		for all $\varphi\in\R$.
	\end{enumerate}
	
	\item[\textbf{(P3)}] For any Banach space $X$, its norm will be denoted by $\|\cdot\|_X$. The symbol $\langle\cdot,\cdot\rangle_{X^*,X}$ denotes the dual pairing of $X$ and its dual space $X^*$. If $X$ is a Hilbert space, its inner product is denoted by $(\cdot,\cdot)_X$.
	
	\item[\textbf{(P4)}] For any $1\le p\le \infty$, $L^p(\Omega)$ and
	$L^p(\Gamma)$ stand for the Lebesgue spaces that are equipped with
	the standard norms $\|\cdot\|_{L^p(\Omega)}$ and
	$\|\cdot\|_{L^p(\Gamma)}$. For $s\ge 0$ and $1\le p\le \infty$, the
	symbols $W^{s,p}(\Omega)$ and $W^{s,p}(\Gamma)$ denote the Sobolev
	spaces with corresponding norms $\|\cdot\|_{W^{s,p}(\Omega)}$ and
	$\|\cdot\|_{W^{s,p}(\Gamma)}$. Note that $W^{0,p}$ can be identified
	with $L^p$. All Lebesgue spaces and Sobolev spaces are Banach spaces
	and if $p=2$, they are Hilbert spaces. In this case we will
	write $H^s(\Omega)=W^{s,2}(\Omega)$ and
	$H^s(\Gamma)=W^{s,2}(\Gamma)$.
	
	\item[\textbf{(P5)}] In general, we will use the symbol $\cdot\vert_\Gamma$ to denote the trace operator. If $1\le p\le \infty$, $s>\frac 1 p$ and $s-\frac 1 p$ is not an integer, the trace operator is uniquely determined and lies in $\mathcal L\big(W^{s,p}(\Omega); W^{s-1/p,p}(\Gamma)\big)$, i.e., it is a linear and bounded operator from $W^{s,p}(\Omega)$ to $W^{s-1/p,p}(\Gamma)$. For brevity, we will sometimes write $\varphi$ instead of $\varphi\vert_\Gamma$ if it is clear that we are referring to the trace of $\varphi$.
	
	\item[\textbf{(P6)}] By $H^1(\Omega)^*$ and $H^1(\Gamma)^*$ we denote the dual spaces of $H^1(\Omega)$ and $H^1(\Gamma)$. For functions $\varphi\in H^1(\Omega)^*$ and $\psi\in H^1(\Omega)^*$ we denote their generalized average by
	\begin{align*}
	\langle \varphi \rangle_\Omega := \frac{1}{|\Omega|} \langle \varphi, 1 \rangle_{H^1(\Omega)^*,H^1(\Omega)}
	\tand
	\langle \psi \rangle_\Gamma := \frac{1}{|\Gamma|} \langle \psi, 1 \rangle_{H^1(\Gamma)^*,H^1(\Gamma)}.
	\end{align*}
	If $\varphi\in L^2(\Omega)$ and $\psi\in L^2(\Gamma)$ the above formulas reduce to
	\begin{align*}
	\langle \varphi \rangle_\Omega = \frac{1}{|\Omega|} \int\limits_\Omega \varphi(x) \dx
	\tand
	\langle \psi \rangle_\Gamma = \frac{1}{|\Gamma|} \int\limits_\Gamma \psi(x) \dSx.
	\end{align*}
	
	\item[\textbf{(P7)}] For any function $\varphi\in H^1(\Omega)^*$ with $\langle\varphi\rangle_\Omega =0$, the Neumann problem
	\begin{align*}
	-\laplace \mu = \varphi \quad\text{in}\; \Omega,\qquad \del_\n \mu = 0 \quad\text{on}\; \Gamma
	\end{align*}
	has a unique weak solution $\mu\in H^1(\Omega)$ with
	$\langle\mu\rangle_\Omega =0$. For any $\psi\in H^1(\Gamma)^*$ with
	$\langle\psi\rangle_\Gamma=0$ the equation
	\begin{align*}
	-\laplace_\Gamma \nu = \psi \quad\text{on}\; \Gamma 
	\end{align*}
	has a unique weak solution $\nu\in H^1(\Gamma)$ with  $\langle\nu\rangle_\Gamma =0$. Here, the symbol $\laplace_\Gamma$ denotes the Laplace-Beltrami operator. We will write
	\begin{align*}
	\nmo\varphi := \mu \tand \nmg\psi:=\nu
	\end{align*}
	to denote the above weak solutions.

	\item[\textbf{(P8)}] In this paper, the spaces $H^1(\Omega)^*$ and $H^1(\Gamma)^*$ will be endowed with the following norms (that are equivalent to the standard norm on these spaces):
	\begin{align*}
	\|\varphi\|_{H^1(\Omega)^*}^2
	&= \big\| \grad\nmo \big(\varphi-\langle \varphi\rangle_\Omega\big) \big\|_{L^2(\Omega)}^2
	+ |\langle \varphi\rangle_\Omega|^2,\\
	\|\psi\|_{H^1(\Gamma)^*}^2
	&= \big\| \grad_\Gamma \nmg \big(\psi-\langle \psi\rangle_\Gamma\big) \big\|_{L^2(\Gamma)}^2 
	+ |\langle \psi\rangle_\Gamma|^2.
	\end{align*} 
	
	\item[\textbf{(P9)}] We define the following sets:
	\begin{align}
	\V &:=	\begin{cases}
	\big\{ \varphi \in H^1(\Omega) \;\big\vert\; \varphi\vert_\Gamma
	\in H^1(\Gamma) \big\}, &\kappa >0,\\
	H^1(\Omega), &\kappa = 0,
	\end{cases}\\
	\Vm &:= \big\{ \varphi \in \V \;\big\vert\; \langle \varphi\rangle_\Omega = m_1\;\text{and}\; \langle\varphi\rangle_\Gamma = m_2 \big\}, \quad m=(m_1,m_2)\in \R^2,\\[4mm]
	\Vo &:= \mathcal V^\kappa_{(0,0)}.
	\end{align}
	Then $\V$ is a Hilbert-space with respect to the inner product
	\begin{align*}
	(\varphi,\psi)_\V := 
	\begin{cases}
	( \varphi , \psi )_{H^1(\Omega)} 
	+ ( \varphi , \psi )_{H^1(\Gamma)}, &\quad\kappa>0,\\
	( \varphi , \psi )_{H^1(\Omega)}, &\quad\kappa=0
	\end{cases}
	\end{align*}
	and its induced norm $\|\varphi\|_\V :=  (\varphi,\varphi)_\V^{1/2}$. 
	
	For $\varphi\in\Va$, $\Va$ being the dual space of $\Vo$, there exists
	a mean value free $u_{\Omega\varphi}\in H^1(\Omega)$ and mean value
	free $u_{\Gamma,\varphi}\in H^1(\Gamma)$ such that for all
	$\zeta\in\Vo$
	\begin{align*}
	\int_\Omega \nabla u_{\Omega,\varphi}\cdot\nabla\zeta \dx 
	+ \int_\Gamma\nabla_\Gamma u_{\Gamma,\varphi} \cdot\nabla_\Gamma \zeta \dS = \varphi(\zeta).
	\end{align*}
	Choosing functions
	$u_{\Omega,\psi}$ and $u_{\Gamma,\psi}$ for $\psi\in \Va$ an analogous way, we can define an inner product on
	$\Va$ by
	\begin{align*}
	\langle \varphi{,}\psi\rangle_{\Va} 
	= \int_\Omega \nabla u_{\Omega,\varphi} \cdot\nabla u_{\Omega,\psi} \dx
	+ \int_\Gamma\nabla_\Gamma u_{\Gamma,\varphi} \cdot\nabla_\Gamma u_{\Gamma,\psi} \dS. 
	\end{align*}
	We also define its induced norm $\|\varphi\|_\Va :=  \langle\varphi,\varphi\rangle_\Va^{1/2}\, $. Since $\Vo\subset\Va$ we can use this inner product and its induced norm also for functions in $\Vo$. For any $\varphi\in \Vo$, the functions $u_{\Omega,\varphi}$ and $u_{\Gamma,\varphi}$ can be expressed by
	\begin{align*}
	u_{\Omega,\varphi} = \nmo\varphi \tand u_{\Gamma,\varphi} = \nmg\varphi\big\vert_\Gamma\,.
	\end{align*}
	It even holds that $\|\cdot\|_\Va$ is a norm on $\Vo$ but, of course, $\Vo$ is not complete with respect to this norm.
	
\end{enumerate}

\begin{remark}\begin{itemize}
		\itemi  In dealing with weak solutions of the system \eqref{CHLW:EPS} (cf. Sections 4 and 5) it suffices to demand that the domain $\Omega$ has merely a Lipschitz-boundary. However, this is not enough to describe \eqref{CHLW:EPS} pointwisely as the Laplace-Beltrami operator is involved. Therefore, a $C^2$-boundary is demanded in the general assumption \textnormal{(P1)} .
		\itemii One can easily see that, according to (P2), the double well potential
		\begin{align*}
		W_\text{dw}(\phi)=\theta(\phi^2-1)^2, \qquad \theta>0
		\end{align*} 
		is a suitable choice for $F$ or $G$. However, the logarithmic
		potential
		\begin{align}
		\label{POT:LOG}
		\begin{aligned}
		W_\text{log}(\phi)= \frac\theta 2 \big((1-\phi)\ln(1-\phi) + (1-\phi)\ln(1-\phi)\big) + \frac{\theta_c}{2}(1-\phi^2),\\
		\quad 0<\theta_c<\theta,
		\end{aligned}
		\end{align}
		(which is defined only for $\phi\in\, ]-1,1[\,$) or the obstacle
		potential
		\begin{align}
		\label{POT:OBST}
		W_\text{obst}(\phi) = 
		\begin{cases}
		\theta(1-\phi^2), & \phi\in [-1,1],\\
		\infty, & \text{else},
		\end{cases}
		\qquad \theta>0,
		\end{align} 
		cannot be chosen as they do not satisfy the condition (P2).
		\itemiii Any nonnegative, convex, continuously differentiable function $\varphi\mapsto F_1(\varphi)$ (or $\varphi\mapsto G_1(\varphi)$ respectively) which grows polynomially to $+\infty$ as
		$|\varphi|\to\infty$ fulfills \textnormal{(P2.3)}. However, exponential growth is not
		allowed, see \cite{garckeelas}.       
	\end{itemize}
\end{remark}

%
%

\section{The gradient flow structure}

For simplicity, we set $\eps=1$. Provided that $\phi$, $\mu$ and $\mu_\Gamma$ are sufficiently regular, we can use the inner product $\langle \cdot,\cdot \rangle_\Va$ that was introduced in (P9) to describe system \eqref{CHLW:EPS} as a gradient flow equation of the total free energy: 
\begin{align}
\label{EQ:GFE}
\langle\delt\phi,\eta\rangle_\Va = -\frac{\delta E}{\delta \phi}(\phi)[\eta] \quad\text{for all}\;  \eta\in \Vo \cap L^\infty(\Omega),\; \eta\vert_\Gamma\in L^\infty(\Gamma). 
\end{align}
This holds because for any  $\eta\in \Vo\cap L^\infty(\Omega)$ with $\eta\vert_\Gamma\in L^\infty(\Gamma)$ , 
\begin{align*}
\langle\delt\phi,\eta\rangle_\Va 
& = \int_\Omega \hspace{-3pt} \grad\nmo\delt\phi\cdot\grad\nmo\eta\dx 
+ \hspace{-3pt}\int_\Gamma\hspace{-3pt}\gradg\nmg\delt\phi\cdot\gradg\nmg\eta\dS \\
& = -\int_\Omega \grad\mu\cdot\grad\nmo\eta\dx 
- \int_\Gamma\gradg\mu_\Gamma\cdot\gradg\nmg\eta\dS
\end{align*}
since $\delt\phi = \laplace\mu$ in $\Omega_T$ and $\delt\phi = \laplace_\Gamma\mu$ on $\Gamma_T$. Integration by parts (recall that $\deln\mu=0$) and the definition of the chemical potentials $\mu$ and $\mu_\Gamma$ imply that
\begin{align*}
\langle\delt\phi,\eta\rangle_\Va 
& = - \int_\Omega \mu\,\eta\dx 
- \int_\Gamma \mu_\Gamma\, \eta\dS \\
& = - \int_\Omega \big(-\laplace\phi + F'(\phi)\big)\eta\dx 
- \int_\Gamma \big(-\kappa\laplace_\Gamma\phi + \deln\phi + G'(\phi)\big)\eta\dS \\
&= -\frac{\delta E}{\delta \phi}(\phi)[\eta].
\end{align*}
This formal computation shows that the gradient flow equation \eqref{EQ:GFE} corresponds to the Cahn--Hilliard
equation with dynamic boundary conditions given by
\eqref{CHLW:EPS}. Formally speaking, we can say that this gradient flow is of type $H^{-1}$ both in the bulk and on the surface.
However, replacing the inner product
$\langle\cdot,\cdot\rangle_\Va$ in the gradient flow equation
\eqref{EQ:GFE} for the energy $E$ by a different inner product leads
to a different PDE in $\Omega$ and different boundary conditions on
$\Gamma$. We give some examples which also can be identified with a
gradient flow equation by a similar computation:
\begin{enumerate}[(i)]
	\item The \textit{Allen--Cahn equation}
	\begin{align*}
	\delt\phi &= \laplace\phi - F'(\phi) \quad\text{in}\;\Omega_T
	\end{align*}
	with the dynamic boundary condition
	\begin{align*}
	\delt\phi = \kappa\laplace_\Gamma\phi - \deln\phi - G'(\phi)  \quad\text{on}\;\Gamma_T
	\end{align*}
	is the gradient flow equation of the energy $E$ with respect to the inner product 
	\begin{align*}
	\langle \phi,\psi \rangle 
	:= \int\limits_\Omega \phi\,\psi\dx
	+ \int\limits_\Gamma \phi\,\psi \dS.
	\end{align*}
	Formally speaking, the gradient flow is of type $L^2$ both in the bulk and on the surface.
	\item The \textit{Allen--Cahn equation}
	\begin{align*}
	\delt\phi &= \laplace\phi - F'(\phi) \quad\text{in}\;\Omega_T
	\end{align*}
	with the dynamic boundary condition
	\begin{align*}
	\left\{
	\begin{aligned}
	\delt\phi &= \laplace_\Gamma \mu_\Gamma &&\text{on}\;\Gamma_T,\\
	\mu_\Gamma &= -\kappa\laplace_\Gamma \phi + \deln\phi + G'(\phi) &&\text{on}\;\Gamma_T
	\end{aligned}
	\right.
	\end{align*}
	is the gradient flow equation of the energy $E$ with respect to the inner product 
	\begin{align*}
	\langle \phi,\psi \rangle 
	:= \int\limits_\Omega \phi\,\psi\dx
	+ \int\limits_\Gamma \gradg\nmg\phi\cdot\gradg\nmg\psi \dS.
	\end{align*}
	This type of system has been analyzed, e.g., in
	\cite{colli-fukao}.
	Formally speaking, the gradient flow is of type $L^2$ in the bulk and of type $H^{-1}$ on the surface. By this boundary condition, the mass on the surface is conserved.
	\item The \textit{Cahn--Hilliard equation}
	\begin{align*}
	\left\{
	\begin{aligned}
	\delt\phi &= \laplace \mu &&\text{in}\;\Omega_T, \\
	\mu &= -\laplace \phi + F'(\phi) &&\text{in}\;\Omega_T
	\end{aligned}
	\right.
	\end{align*}
	with the homogeneous Neumann condition $\deln\mu = 0$ on $\Gamma_T$ and the dynamic boundary condition
	\begin{align*}
	\delt\phi = \kappa\laplace_\Gamma\phi - \deln\phi - G'(\phi) \quad\text{on}\;\Gamma_T
	\end{align*}
	is the gradient flow equation of the energy $E$ with respect to the inner product
	\begin{align*}
	\langle \phi,\psi \rangle := \int\limits_\Omega \grad\nmo\phi\cdot\grad\nmo\psi\dx
	+ \int\limits_\Gamma \phi\,\psi \dS.
	\end{align*}
	This problem is introduced in \cite{kenzler} and analyzed in \cite{colli-sprekels}. Formally speaking, the gradient flow is of type $H^{-1}$ in the bulk and of type $L^2$ on the surface. Note that the boundary condition $\deln\mu = 0$ leads to mass conservation in the bulk.
	\item Let us now consider the elliptic system
	\begin{align}
	\label{PDE:WENZ}
	\left\{
	\begin{aligned}
	\laplace u &= f_1 &&\text{in}\;\Omega, \\
	\sigma \laplace_\Gamma u - \deln u &= f_2 &&\text{on}\;\Gamma. \\
	\end{aligned}
	\right.
	\end{align}
	Using the Lax-Milgram theorem one can show that the system
	\eqref{PDE:WENZ} with $f=(f_1,f_2)$ has a unique weak solution
	$u=\mathcal S(f)$ with $\langle u\rangle_\Omega = 0$ if the right-hand side $f$
	satisfies the conditions
	$\langle f_1\rangle_\Omega +\langle f_2\rangle_\Gamma =
	0$. This means that we can define a solution operator
	$f\mapsto \mathcal S(f)$ that maps any admissible right-hand
	side $f$ onto its corresponding solution. \\[1ex]
	Then, the \textit{Cahn--Hilliard equation}
	\begin{align}
	\label{PDE:CHW}
	\left\{
	\begin{aligned}
	\delt\phi &= \laplace \mu &&\text{in}\;\Omega_T, \\
	\mu &= -\laplace \phi + F'(\phi) &&\text{in}\;\Omega_T
	\end{aligned}
	\right.
	\end{align}
	with the dynamic boundary condition
	\begin{align}
	\label{PDE:BCW}
	\left\{
	\begin{aligned}
	\delt\phi &= \sigma\laplace_\Gamma\mu - \deln\mu &&\text{on}\;\Gamma_T\\
	\mu 
	&= -\kappa \laplace_\Gamma\phi + \deln\phi + G'(\phi) &&\text{on}\;\Gamma_T
	\end{aligned}
	\right.
	\end{align}
	(for some parameter $\sigma\ge 0$) is the gradient flow equation of the energy $E$  with respect to the inner product
	\begin{align*}
	\langle \phi,\psi \rangle := \int\limits_\Omega \grad \mathcal S(\phi)\cdot\grad \mathcal S(\psi)\dx
	+ \sigma \int\limits_\Gamma \gradg \mathcal S(\phi)\cdot\gradg \mathcal S(\psi)\dS.
	\end{align*}
	This is the problem introduced in
	\cite{Gal1,goldstein} which reduces to the Wentzell boundary condition for
	$\sigma=0$ (see \cite{colli-fukao,Gal1,GalWu}).
	In this model, the total mass (i.e., the sum of bulk and surface mass) is conserved. 	Note that integrating and adding the second lines of \eqref{PDE:CHW} and \eqref{PDE:BCW} implies that 
	\begin{align*}
	\int\limits_\Omega \mu \dx + \int\limits_\Gamma \mu \dS = \int\limits_\Omega F'(\phi) \dx + \int\limits_\Gamma G'(\phi) \dS
	\end{align*}
	must hold. Therefore, $\mu$ cannot be equal to $S(\delt\phi)$ but instead $\mu = S(\delt\phi)+c$ where the constant is given by
	\begin{align*}
	c:= \frac{\int_\Omega F'(\phi)\dx + \int_\Gamma G'(\phi)\dS}{|\Omega|+|\Gamma|}.
	\end{align*}
	Then $\mu$ is still a solution of \eqref{PDE:WENZ} with $f_1=\delt\phi$ and $f_2=\delt\phi\vert_\Gamma$ and the inner product is not
	affected by this shift as only the gradients of the
	solution components are involved. 
\end{enumerate}
In the next section we will exploit the gradient flow structure \eqref{EQ:GFE} to construct a weak solution of the Cahn--Hilliard system \eqref{CHLW:EPS} by implicit time discretization. Certainly, weak solutions of the above systems (i)-(iv) could be constructed in a similar fashion.

%

\section{Existence of a weak solution}

As stated in the introduction, we consider the following Cahn--Hilliard system with a dynamic boundary condition with $\eps=1$: 
\begin{subequations}
	\label{CHLW}
	\begin{align}
	&\delt\phi = \laplace \mu \quad\twith \mu = -\laplace \phi + F'(\phi)
	&& \text{in} \;\; \Omega\times]0,T[, \\	
	&\del_\n \mu = 0 
	&& \text{on} \;\; \Gamma\times]0,T[, \\	
	&\delt\phi = \laplace_\Gamma \mu_\Gamma
	\twith \mu_\Gamma = -\kappa\laplace_\Gamma \phi + \del_\n \phi + G'(\phi)
	&& \text{on} \;\; \Gamma\times]0,T[, \\	
	&\phi(\cdot,0) = \phi_0
	&& \text{in} \;\; \Omega.
	\end{align}
\end{subequations}
The choice $\eps=1$ means no loss of generality as it does not play any role in the analysis.

\subsection{Weak solutions and the existence theorem}

Before formulating the existence theorem we give the definition of a weak solution of the Cahn--Hilliard equation \eqref{CHLW}. In the following, as we only consider weak solutions, it suffices to assume that $\Omega$ has merely a Lipschitz-boundary.

\begin{definition}
	\label{DEF:WS}
	Let $T\in]0,\infty[$, $\kappa\ge 0$ and let $\phi_0\in \Vm$ be any initial datum having finite energy, i.e., $E(\phi_0)<\infty$. Then the triple $(\phi,\mu,\mu_\Gamma)$ is called a weak solution of the system \eqref{CHLW} if the following holds:
	\begin{itemize}
		\itemi The occuring functions have the following regularity:
		\begin{alignat*}{3}
		&\phi \in C\big([0,T];L^2(\Omega)\big) 
		\cap L^\infty\big(0,T;H^1(\Omega)\big), 
		&&\mu \in L^2\big(0,T;H^1(\Omega)\big),\\
		&\phi\vert_\Gamma \in C\big([0,T];L^2(\Gamma)\big) \cap 
		\begin{cases}
		L^\infty\big(0,T;H^1(\Gamma)\big), &\hspace{-5pt}\kappa >0,\\
		L^\infty\big(0,T;H^{1/2}(\Gamma)\big) , &\hspace{-5pt}\kappa =0,
		\end{cases}\;\;
		&&\mu_\Gamma \in L^2\big(0,T;H^1(\Gamma)\big),\\
		&F'(\phi) \in L^1(\Omega_T), 
		&&G'(\phi) \in L^1(\Gamma_T).
		\end{alignat*}
		\itemii The following weak formulations are satisfied:
		\begin{align}
		\label{DEF:WF1}
		\int\limits_{\Omega_T} (\phi-\phi_0)\, \delt\zeta\dxt
		&= 	\int\limits_{\Omega_T} \grad\mu \cdot \grad\zeta \dxt
		\end{align}
		for all $\zeta\in L^2\big(0,T;H^1(\Omega)\big)$ with $\delt\zeta\in L^2(\Omega_T)$ and $\zeta(T)=0$,
		\begin{align}
		\label{DEF:WF2}
		\int\limits_{\Gamma_T} (\phi-\phi_0)\, \delt\xi\dxt
		&= 	\int\limits_{\Gamma_T} \gradg \mu_\Gamma \cdot \gradg \xi \dxt
		\end{align}
		for all $\xi\in L^2\big(0,T;H^1(\Gamma)\big)$ with $\delt\xi\in L^2(\Gamma_T)$ and $\xi(T)=0$ and 
		\begin{align}
		\label{DEF:WF3}
		\begin{aligned}
		\int\limits_{\Omega_T} \mu\,\eta\dx + \int\limits_{\Gamma_T} \mu_\Gamma\,\eta \dS\mathrm dt  
		&= \int\limits_{\Omega_T} \grad\phi\cdot\grad\eta \dxt 
		+ \int\limits_{\Omega_T} F'(\phi)\,\eta\dxt \notag\\
		&\quad +\int\limits_{\Gamma_T} \kappa \gradg \phi \cdot \gradg \eta \dS\mathrm dt
		+\int\limits_{\Gamma_T} G'(\phi) \eta \dS\mathrm dt
		\end{aligned}
		\end{align}
		for all $\eta\in L^2\big(0,T;\V\big) \cap L^\infty(\Omega_T)$ with $\eta\vert_\Gamma\in L^\infty(\Gamma)$.
		\itemiii The energy inequality is satisfied, i.e., for all $t\in[0,T]$,
		\begin{align}
		E\big(\phi(t)\big) 
		+ \frac 1 2 \int\limits_0^t \|\grad\mu(s)\|_{L^2(\Omega)}^2 + \|\gradg\mu_\Gamma(s)\|_{L^2(\Gamma)}^2 \ds
		\le E(\phi_0) < \infty.
		\end{align}
	\end{itemize}
\end{definition}

\begin{remark}
	\label{RMK:WS}
	Let us assume that $(\phi,\mu,\mu_\Gamma)$ is a weak solution in the sense of Definition \ref{DEF:WS}.
	\begin{enumerate}
		\item[\textnormal{(a)}] It follows from the weak formulations \eqref{DEF:WF1} and \eqref{DEF:WF2} that
		\begin{align}
		\label{ASS:REM1}
		\langle \phi(t) \rangle_\Omega &= \langle \phi_0 \rangle_\Omega = m_1,\\
		\label{ASS:REM2}
		\langle \phi(t) \rangle_\Gamma &= \langle \phi_0 \rangle_\Gamma = m_2
		\end{align}
		for all $t\in[0,T]$. To prove this, let $t_0\in]0,T]$ and $0<h<t_0/2$ be arbitrary and choose
		\begin{align*}
		\zeta(t) := \begin{cases}
		1, &t< t_0-h, \\
		 (t_0-t)/h, & t_0-h \le t \le t_0,\\
		0, & t > t_0.
		\end{cases}
		\end{align*}
		Then $\zeta$ is a suitable test function with $\grad\zeta = 0$ and $\delt\zeta = -\frac 1 h \chi_{[t_0-h,t_0]}$. Plugging $\zeta$ into \eqref{DEF:WF1} yields
		\begin{align*}
		\frac 1 h \int_{t_0-h}^{t_0} \int_{\Omega} \phi(t) \dx \dt = \int_{\Omega} \phi_0 \dx.
		\end{align*}
		Since $t\mapsto \int_\Omega \phi(t) \dx$ is continuous, \eqref{ASS:REM1} follows as $h$ tends to zero. The assertion \eqref{ASS:REM2} can be proved analogously.
		\item[\textnormal{(b)}] Since $\mu\in
		L^2\big(0,T;H^1(\Omega)\big)$ and $\mu_\Gamma\in
		L^2\big(0,T;H^1(\Gamma)\big)$, it also follows (by
		definition) from the weak formulations \eqref{DEF:WF1} and
		\eqref{DEF:WF2} that $\phi$ and $\phi\vert_\Gamma$ have
		generalized derivatives with respect to $t$ given by
		\begin{align*}
		\delt\phi = \laplace\mu \;\;\text{in}\;\; L^2\big(0,T;H^1(\Omega)^*\big) , \quad \delt\phi\vert_\Gamma = \laplace_\Gamma \mu_\Gamma \;\;\text{in}\;\; L^2\big(0,T;H^1(\Gamma)^*\big).
		\end{align*}
	\end{enumerate}
\end{remark}

The following theorem constitutes the main result of this paper: The existence of a weak solution of the Cahn--Hilliard system \eqref{CHLW}.

\begin{theorem}
	\label{MAIN}
	We assume that $\Omega\subset\R^d$ (with $d\in\{2,3\}$) is a bounded domain with Lipschitz boundary and that the potentials $F$ and $G$ satisfy the condition \textnormal{(P2)}. Let $T>0$ and $m=(m_1,m_2)\in \R^2$ be arbitrary and let $\phi_0 \in \Vm$ be any initial datum having finite energy, i.e., $E(\phi_0)<\infty$. Then there exists a weak solution $(\phi,\mu,\mu_\Gamma)$ of the initial value problem \eqref{CHLW} in the sense of Definition \ref{DEF:WS}. This solution has the following additional properties:
	\begin{align*}
	\phi &\in C^{0,1/4}\big([0,T];L^2(\Omega)\big), \\
	\phi\vert_\Gamma &\in 
	\begin{cases}
	C^{0,1/4}\big([0,T];L^2(\Gamma)\big), &\hspace{-5pt}\kappa >0,\\
	C^{0,1/4}\big([0,T];H^1(\Gamma)^*\big) , &\hspace{-5pt}\kappa =0.
	\end{cases}\\
	\end{align*}
\end{theorem}

\subsection{Implicit time discretization}
To prove Theorem \ref{MAIN}, some preparation is
necessary. The first step is to derive an implicit time discretization of the system \eqref{CHLW}. After that, we intend to show that the corresponding time-discrete solution converges to a weak solution of \eqref{CHLW} in some suitable sense. This is a common approach in dealing with gradient flow equations that has already been used extensively in the literature (see, e.g., \cite{ambrosio}).\\[1ex]
To this end, let $N\in\N$ be arbitrary and let $\tau:=T/N$
denote the time step size. Without loss of generality, we assume that $\tau\le 1$. Now, we define $\phi^n$, $ n=0,...,N$ recursively by the following construction: The $0$-th iterate is the initial datum, i.e., $\phi^0:=\phi_0$. If the $n$-th iterate $\phi^n$ is already constructed, we choose $\phi^{n+1}$ to be a minimizer of the functional
\begin{align*}
J_n(\phi):=\frac{1}{2\tau}\|\phi-\phi^n\|_\Va^2 + E(\phi)
\end{align*}
on the set $\Vm$.  Note that $J_n$ may attain the value $+\infty$.  The existence of such a minimizer is guaranteed by Lemma \ref{LEM:EXM} (that will be established in the next subsection). The idea behind this definition becomes clear when considering the first variation of the functional $J_n$ at the point $\phi^{n+1}$. As $\phi^{n+1}$ is a minimizer  and since $F_1$ and $G_1$ are convex, we can proceed as in \cite[Lem.\,3.2]{garckeelas} to conclude that 
\begin{align*}
0 &= \Big\langle \frac{\phi^{n+1}-\phi^n}{\tau} , \eta \Big\rangle_\Va
+\int\limits_\Omega \grad\phi^{n+1}\cdot\grad\eta \dx + \int\limits_\Omega F'(\phi^{n+1})\,\eta\dx \\
&\qquad +\int\limits_\Gamma \kappa \gradg \phi^{n+1} \cdot \gradg \eta \dSx
+\int\limits_\Gamma G'(\phi^{n+1}) \eta \dSx \notag	
\end{align*}
 for all directions $\eta \in \Vo \cap L^\infty(\Omega)$ with $\eta\vert_\Gamma\in L^\infty(\Gamma)$. This can be interpreted as an implicit time discretization of the corresponding gradient flow equation \eqref{EQ:GFE}.  We set
\begin{align*}
\mathring\mu^{n+1} := \nmo\left(-\frac{\phi^{n+1}-\phi^n}{\tau}\right) 
\tand 
\mathring\mu_\Gamma^{n+1} := \nmg\left(-\frac{\phi^{n+1}\vert_\Gamma-\phi^n\vert_\Gamma}{\tau}\right).
\end{align*}
According to (P7), this means that $\mathring\mu^{n+1} \in H^1(\Omega)$ and $\mathring\mu_\Gamma^{n+1} \in H^1(\Gamma)$ are a solution of the Poisson equations
\begin{align}
\label{PDE:MUO}
-\laplace\mu &= -\frac{\phi^{n+1}-\phi^n}{\tau} \quad \text{in}\; \Omega 
\twith \del_\n \mu = 0 \quad \text{on}\; \Gamma, \\
\label{PDE:MUG}
-\laplace_\Gamma\mu_\Gamma &= -\frac{\phi^{n+1}-\phi^n}{\tau} \quad \text{on}\; \Gamma
\end{align} 
with $\langle\mathring\mu^{n+1} \rangle_\Omega = 0$ and $\langle\mathring\mu^{n+1} \rangle_\Gamma = 0$.
Hence, the functions $\phi^{n+1},\;\phi^n,\;\mathring\mu^{n+1}$ and $\mathring\mu_\Gamma^{n+1}$ satisfy the equation
\begin{align}
\label{EQ:WFV0}
\begin{aligned}
\int\limits_\Omega \mathring\mu^{n+1}\,\eta\dx + \int\limits_\Gamma \mathring\mu_\Gamma^{n+1}\,\eta \dSx
&= \int\limits_\Omega \grad\phi^{n+1}\scdot\grad\eta + F'(\phi^{n+1})\,\eta\dx \\ 
&\quad +\int\limits_\Gamma \kappa \gradg \phi^{n+1} \scdot \gradg \eta + G'(\phi^{n+1}) \eta \dSx  
\end{aligned}
\end{align}
for all functions  $\eta \in \Vo \cap L^\infty(\Omega)$ with $\eta\vert_\Gamma\in L^\infty(\Gamma)$  and all $n\in\{0,...,N-1\}$. However, to obtain an approximate solution of the system \eqref{CHLW}, we need equation \eqref{EQ:WFV0} to hold for all test functions  $\eta \in \V \cap L^\infty(\Omega)$ with $\eta\vert_\Gamma\in L^\infty(\Gamma)$.  The idea is to replace $\mathring\mu^{n+1}$ and $\mathring\mu_\Gamma^{n+1}$ by $$\mu^{n+1}:=\mathring\mu^{n+1} + c^{n+1} \tand \mu_\Gamma^{n+1}:=\mathring\mu_\Gamma^{n+1} + c_\Gamma^{n+1},$$ 
with constants $c^{n+1},c_\Gamma^{n+1}\in\R$ that do not depend on $\eta$, since these functions $\mu^{n+1}$ and $\mu_\Gamma^{n+1}$ are still a solution of  \eqref{PDE:MUO} and \eqref{PDE:MUG} . We will now show that \eqref{EQ:WFV0} holds true for all  $\eta \in \V \cap L^\infty(\Omega)$ with $\eta\vert_\Gamma\in L^\infty(\Gamma)$  if the functions $\mathring\mu^{n+1}$ and $\mathring\mu_\Gamma^{n+1}$ are replaced by $\mu^{n+1}$ and $\mu_\Gamma^{n+1}$ with suitably chosen constants $c^{n+1}$ and $c_\Gamma^{n+1}$. Therefore, let  $\eta \in \V \cap L^\infty(\Omega)$ with $\eta\vert_\Gamma\in L^\infty(\Gamma)$  be arbitrary. We define a new test function $\eta_0$  by
\begin{align*}
\eta_0 := \eta - c_1\eta_1 - c_2\eta_2
\end{align*}
where $c_1,c_2\in \R$, $\eta_1\equiv 1$ and $\eta_2\in C^\infty_c(\Omega)$ is an arbitrary nonnegative function that is not identically zero. Of course, this means that  $\eta_0 \in \V \cap L^\infty(\Omega)$ with $\eta\vert_\Gamma\in L^\infty(\Gamma)$.  Choosing 
\begin{align*}
c_1:=\frac{1}{|\Gamma|}\int\limits_\Gamma \eta \dSx \tand c_2 := \left(\int\limits_\Omega\eta \dx - \frac{|\Omega|}{|\Gamma|} \int\limits_\Gamma \eta \dSx\right) \Bigg/ \left(\int\limits_\Omega \eta_2 \dx\right)
\end{align*}
we obtain that
\begin{gather*}
\int\limits_\Gamma \eta_0 \dSx = \int\limits_\Gamma \eta \dSx - c_1 |\Gamma| = 0, \\
\int\limits_\Omega \eta_0 \dx 
= \int\limits_\Omega \eta \dx - c_1 |\Omega| -c_2 \int\limits_\Omega \eta_2 \dx 
= 0.
\end{gather*}
Thus, $\eta_0$ even lies in  $\Vo \cap L^\infty(\Omega)$ with $\eta\vert_\Gamma\in L^\infty(\Gamma)$  and may thus be plugged into equation \eqref{EQ:WFV0}. This yields
\begin{align*}
\int\limits_\Omega \grad\phi^{n+1}\cdot\grad(\eta -c_2\eta_2) 
+  F'(\phi^{n+1})\,(\eta-c_1-c_2\eta_2)\dx 
+\int\limits_\Gamma \kappa \gradg \phi^{n+1} \cdot \gradg \eta \dSx \\
+\int\limits_\Gamma G'(\phi^{n+1}) (\eta-c_1) \dSx 
= \int\limits_\Omega \mathring\mu^{n+1}\,(\eta-c_2\eta)\dx + \int\limits_\Gamma \mathring\mu_\Gamma^{n+1}\,\eta \dSx
\end{align*}
which is equivalent to
\begin{align*}
&\int\limits_\Omega \grad\phi^{n+1}\cdot\grad\eta + F'(\phi^{n+1})\,\eta\dx 
+\int\limits_\Gamma \kappa \gradg \phi^{n+1} \cdot \gradg \eta + G'(\phi^{n+1}) \eta \dSx \notag\\
&= \int\limits_\Omega \mathring\mu^{n+1}\,\eta\dx + \int\limits_\Gamma \mathring\mu_\Gamma^{n+1}\,\eta \dSx 
+c_1 \int\limits_\Omega F'(\phi^{n+1})\dx +c_1 \int\limits_\Gamma G'(\phi^{n+1}) \dSx \\
&\quad- c_2 \int\limits_\Omega \mathring\mu^{n+1}\,\eta_2\dx + c_2 \int\limits_\Omega \grad \phi^{n+1}\cdot \grad \eta_2 \dx
+ c_2\int\limits_\Omega F'(\phi^{n+1}) \eta_2 \dx.
\end{align*}
By the definition of $c_1$ and $c_2$ we obtain that
\begin{align*}
\begin{aligned}
&\int\limits_\Omega \grad\phi^{n+1}\cdot\grad\eta + F'(\phi^{n+1})\,\eta\dx 
+\int\limits_\Gamma \kappa \gradg \phi^{n+1} \cdot \gradg \eta + G'(\phi^{n+1}) \eta \dSx \notag\\
&\qquad = \int\limits_\Omega (\mathring\mu^{n+1} + c^{n+1})\,\eta\dx + \int\limits_\Gamma (\mathring\mu_\Gamma^{n+1} + c_\Gamma^{n+1})\,\eta \dSx 
\end{aligned}
\end{align*}
where the constants $c^{n+1}$ and $c_\Gamma^{n+1}$ are defined by
\begin{align*}
c^{n+1}&:= \left(\int\limits_\Omega -\mathring\mu^{n+1}\eta_2 + \grad\phi^{n+1}\cdot\grad\eta_2 + F'(\phi^{n+1})\eta_2 \dx \right)\Bigg/ \left(\int\limits_\Omega \eta_2 \dx \right),\\
c_\Gamma^{n+1}&:= \left(|\Omega|\int\limits_\Omega -\mathring\mu^{n+1}\eta_2 + \grad\phi^{n+1}\cdot\grad\eta_2 + F'(\phi^{n+1})\eta_2 \dx \right)\Bigg/ \left(|\Gamma|\int\limits_\Omega \eta_2 \dx \right)\\
&\qquad + \frac{1}{|\Gamma|} \left( \int\limits_\Omega F'(\phi^{n+1}) \dx + \int\limits_\Gamma G'(\phi^{n+1})\dSx  \right).
\end{align*}
Consequently the functions $\phi^{n+1},\;\phi^n,\;\mu^{n+1}$ and $\mu_\Gamma^{n+1}$ are a solution of
\begin{subequations}
	\begin{align}
	\label{EQ:WSD1}
	&\langle \grad\mu^{n+1},\grad\zeta \rangle_{L^2(\Omega)}  
	= -\Big\langle \frac{\phi^{n+1}-\phi^n}{\tau}, \zeta \Big\rangle_{L^2(\Omega)}
	\quad\text{for all}\; \zeta\in H^1(\Omega), 
	\\
	\label{EQ:WSD2}
	&\langle \gradg\mu_{\Gamma}^{n+1},\gradg\xi \rangle_{L^2(\Gamma)}  
	= -\Big\langle \frac{\phi^{n+1}-\phi^n}{\tau}, \xi \Big\rangle_{L^2(\Gamma)}
	\quad\text{for all}\; \xi\in H^1(\Gamma), 
	\\[2mm]
	\label{EQ:WSD3}
	&\begin{aligned}
	&\int\limits_\Omega\hspace{-3pt} \grad\phi^{n+1}\scdot\grad\eta + F'(\phi^{n+1})\,\eta\dx 
	+ \int\limits_\Gamma \hspace{-3pt} \kappa \gradg \phi^{n+1} \scdot \gradg \eta +  G'(\phi^{n+1}) \eta \,\mathrm dS(x) \\
	&\; = \int\limits_\Omega \hspace{-3pt} \mu^{n+1}\,\eta\dx 
	+ \int\limits_\Gamma \mu_\Gamma^{n+1}\,\eta \dSx  	\quad\text{for all}\;  \eta\in \V \cap L^\infty(\Omega), \eta\vert_\Gamma \in L^\infty(\Gamma)),
	\end{aligned}
	\end{align}
\end{subequations}
which is an implicit time discretization of the system \eqref{CHLW}. This means that the triple $(\phi^n,\mu^n,\mu_\Gamma^n)$, $n=1,...,N$ describes a time-discrete approximate solution. \\[1ex]
In the following, $(\phi_N,\mu_N,\mu_{\Gamma,N})$ will denote the \textit{piecewise constant extension} of the approximate solution $(\phi^n,\mu^n,\mu_\Gamma^n)_{n=1,...,N}$ on the interval $[0,T]$, i.e., for $n\in\{1,...,N\}$ and $t\in](n-1)\tau,n\tau]$, we set\vspace{-2mm}
\begin{align}
(\phi_N,\mu_N,\mu_{\Gamma,N})(\cdot,t):=(\phi_N^n,\mu_N^n,\mu_{\Gamma,N}^n):=(\phi^n,\mu^n,\mu_\Gamma^n).
\end{align}
Similarly, we define the \textit{piecewise linear extension} $(\bar\phi_N,\bar\mu_N,\bar\mu_{\Gamma,N})$ by
\begin{align}
(\bar\phi_N,\bar\mu_N,\bar\mu_{\Gamma,N})(\cdot,t)
:=\alpha (\phi^n_N,\mu^n_N,\mu_{\Gamma,N}^n) + (1-\alpha) (\phi^{n-1}_N,\mu^{n-1}_N,\mu_{\Gamma,N}^{n-1})
\end{align}
for $n\in\{1,...,N\}$, $\alpha\in[0,1]$ and $t=\alpha n \tau + (1-\alpha)(n-1)\tau$.

\subsection{The functional $J_n$ has a minimizer on $\Vm$}
We now show existence of a time discrete solution using methods from
calculus of variations.

\begin{lemma}
	\label{LEM:EXM}
	Let $N\in\N$ and $\tau>0$ as defined in Section 4.2 and let $n\in\{1,...,N\}$ be arbitrary. Then the functional
	\begin{align*}
	J_n(\phi):=\frac{1}{2\tau}\|\phi-\phi^n\|_\Va^2 + E(\phi),\quad \phi\in \Vm
	\end{align*}
	has a global minimizer $\bar\phi\in\Vm$,
	i.e., 
	\begin{align*}
	J_n(\bar\phi)\le J_n(\phi) , \quad \mbox{ for all } \phi\in\Vm.
	\end{align*} 
\end{lemma}

\begin{proof}
	We can prove the existence of a minimizer by the direct method of calculus of variations. 
	Obviously, $J$ is bounded from below by
	\begin{align*}
	J_n(\phi) \ge \int\limits_\Omega F(\phi) \dx + \int\limits_\Gamma G(\phi) \dS 
	\ge |\Omega|C_F + |\Gamma|C_G, \qquad\text{for all}\; \phi\in\Vm.
	\end{align*}
	Thus, the infimum $M:= \inf_\Vm J_n$ exists and therefore we can find a minimizing sequence $(\phi_k)_{k\in\N}\subset \Vm$ with
	\begin{align*}
	J_n(\phi_k)\to M,\quad k\to\infty \quad\tand\quad J_n(\phi_k) \le M+1,\quad k\in\N.
	\end{align*}
	From the definition of $J$, we conclude that
	\begin{align*}
	\frac 1 2 \|\grad\phi_k\|_{L^2(\Omega)}^2 
	+ \frac \kappa 2 \|\grad_\Gamma \phi_k\|_{L^2(\Gamma)}^2 
	&\le J_n(\phi_k) - \int\limits_\Omega F(\phi_k) \dx - \int\limits_\Gamma G(\phi_k) \dS \\
	&\le M + 1 - C_F|\Omega| - C_G|\Omega|
	\end{align*}
	for all $k\in\N$. Since $\langle\phi_k\rangle_\Omega = m_1$ and $\langle\phi_k\rangle_\Gamma = m_2$ for all $k\in\N$, we can use Poincar\'e's inequality to infer that $(\phi_k)_{k\in\N}$ is bounded in the Hilbert space $\V$. Hence, the Banach-Alaoglu theorem implies that there exists some function $\bar\phi \in\V$ such that $\phi_k \wto \bar\phi$ in $\V$ after extraction of a subsequence. Then
	\begin{align*}
	\big| \langle\bar\phi\rangle_\Omega - m_1 \big| = \big| \langle\bar\phi\rangle_\Omega - \langle \phi_k\rangle_\Omega \big| = \frac{1}{|\Omega|}\Bigg|\int\limits_\Omega \bar\phi - \phi_k \dx\Bigg| \to 0,\quad k\to\infty
	\end{align*} 
	since $\phi_k \wto \bar\phi$ especially in $L^2(\Omega)$. This means that $\langle\bar\phi\rangle_\Omega = m_1$. The equality $\langle\bar\phi\rangle_\Gamma = m_2$ can be proved analogously due to the compact embedding $H^1(\Omega)\hookrightarrow L^2(\Gamma)$. As $H^1(\Omega)$ is compactly embedded in $L^2(\Omega)$ we obtain, after another subsequence extraction, that 
	\begin{align*}
	\phi_k \to \bar\phi\quad\text{in}\;L^2(\Omega) \quad\tand\quad \phi_k \to \bar\phi\quad\text{almost everywhere in}\;\Omega
	\end{align*}
	as $k\to\infty$. From the compact embedding $H^1(\Omega)\hookrightarrow L^2(\Gamma)$ we conclude that
	\begin{align*}
	\phi_k \to \bar\phi\quad\text{in}\;L^2(\Gamma) \quad\tand\quad \phi_k \to \bar\phi\quad\text{almost everywhere in}\;\Gamma
	\end{align*}
	up to a subsequence. Obviously, all terms in the functional $J$ are convex in $\phi$ besides $\int_\Omega F_2(\phi) \mathrm dx$ and $\int_\Gamma G_2(\phi) \mathrm dS$. We know that $F_2(\phi_k)\to F_2(\bar\phi)$ almost everywhere in $\Omega$ and $G_2(\phi_k)\to G_2(\bar\phi)$ almost everywhere on $\Gamma$. Assumption (P2.4) implies that
	\begin{align*}
	|F_2(\phi_k)| \le C \big(1+|\phi_k|^2\big) \quad\text{in}\;\Omega
	\tand
	|G_2(\phi_k)| \le C \big(1+|\phi_k|^2\big) \quad\text{on}\;\Gamma
	\end{align*}
	for some constant $C\ge 0$ depending only on $B_F$ and $B_G$. Hence Lebesgue's general convergence theorem (cf. \cite[p.\,60]{alt}) implies that
	\begin{align*}
	\int\limits_\Omega F_2\big(\phi_k(x)\big) \dx 
	\to \int\limits_\Omega F_2\big(\bar\phi(x)\big) \dx,
	\quad
	\int\limits_\Gamma G_2\big(\phi_k(x)\big) \dSx 
	\to \int\limits_\Gamma G_2\big(\bar\phi(x)\big) \dSx
	\end{align*}
	as $k\to\infty$. Then, as the convex terms of $J$ are weakly lower semicontinuous, we obtain that
	\begin{align*}
	J_n(\bar\phi) \le \underset{k\to\infty}{\lim\inf} J_n(\phi_k) = M
	\end{align*}
	which directly yields $J_n(\bar\phi)=M$ by the definition of $M$.
\end{proof}

\subsection{Uniform bounds on the piecewise constant extension}
In this section the gradient flow structure will allow us to establish
uniform a priori estimates.

\begin{lemma}
	\label{LEM:BND}
	There exist nonnegative constants $C_1,...,C_5$ that do not depend on $N$, $n$ or $\tau$ such that 
	\begin{align}
	\label{BND:PHIN}
	&\|\phi_N\|_{L^\infty(0,T;H^1(\Omega))} \le C_1,\\
	\label{BND:PHING}
	&\begin{cases} 
	\|\phi_N\|_{L^\infty(0,T;H^{1/2}(\Gamma))} \le C_2 &\text{if}\; \kappa=0, \\
	\|\phi_N\|_{L^\infty(0,T;H^1(\Gamma))} \le C_3 &\text{if}\; \kappa>0, 
	\end{cases}\\[2mm] 
	\label{BND:MUN}
	&\|\mu_N\|_{L^2(0,T;H^1(\Omega))} \le C_4,\\
	\label{BND:MUNG}
	&\|\mu_{\Gamma,N}\|_{L^2(0,T;H^1(\Gamma))} \le C_5
	\end{align}
	for all $N\in\N$.
\end{lemma}

\begin{proof}
	In the following, the letter $C$ will denote a generic nonnegative constant that does not depend on $n$, $\tau$ or $N$ and may change its value from line to line. First, as for any ${n\in\{0,...,N-1\}}$, $\phi^{n+1}$ was chosen to be a minimizer of the functional $J_n$ on the set $\Vm$, we obtain the a priori bound
	\begin{align}
	\label{IEQ:APB}
	\frac{1}{2\tau}\|\phi^{n+1}-\phi^n\|_\Va^2 + E(\phi^{n+1}) = J_n(\phi^{n+1}) \le J_n(\phi^n)=E(\phi^n)
	\end{align}
	for all $n\in\{0,...,N-1\}$. It follows inductively that 
	\begin{align}
	\label{IEQ:APB2}
	E(\phi^{n+1})\le E(\phi_0), \quad n\in\{0,...,N-1\}.
	\end{align}
	From the definition of $E$ we conclude the uniform bound
	\begin{align}
	\label{EST:PHI}
	\begin{aligned}
	&\frac 1 2 \|\grad\phi^{n+1}\|_{L^2(\Omega)} + \frac \kappa 2 \|\gradg \phi^{n+1}\|_{L^2(\Gamma)} \le E(\phi_0) + C^*, \\
	&\text{with}\;\; C^*:=- C_F|\Omega| - C_G|\Gamma| , 
	\end{aligned}
	\end{align}
	for all $n\in\{0,...,N-1\}$. Recall that $\langle \phi^{n+1}\rangle_\Omega = m_1$ and $\langle \phi^{n+1}\rangle_\Gamma = m_2$. Thus, by Poincar\'e's inequality,
	\begin{align}
	\label{EST:PHI3}
	\begin{aligned}
	\|\phi^{n+1}\|_{L^2(\Omega)} &\le \|\langle \phi^{n+1} \rangle_\Omega \|_{L^2(\Omega)} 
	+ \|\phi^{n+1} - \langle \phi^{n+1} \rangle_\Omega \|_{L^2(\Omega)}\\
	& \le |\Omega|^{\frac 1 2}m_1 + C\,\|\grad \phi^{n+1}\|_{L^2(\Omega)} \le C
	\end{aligned}
	\end{align}
	and, as the corresponding trace operator lies in $\mathcal L\big(H^1(\Omega);H^{1/2}(\Gamma)\big)$, we also have
	\begin{align}
	\label{EST:PHI4}
	\|\phi^{n+1}\|_{H^{1/2}(\Gamma)} \le C\,\|\phi^{n+1}\|_{H^{1}(\Omega)}.
	\end{align}
	This means that $\phi_N$ is uniformly bounded in
	$L^\infty\big(0,T;H^1(\Omega)\big)$ and its trace $\phi_N\vert_\Gamma$ is
	uniformly bounded in $L^\infty\big(0,T;H^{1/2}(\Gamma)\big)$. If
	$\kappa>0$, the trace $\phi_N\vert_\Gamma$ is even uniformly bounded
	in $L^\infty\big(0,T;H^1(\Gamma)\big)$ according to
	\eqref{EST:PHI}. Recalling the definition of $\phi_N$, this proves \eqref{BND:PHIN} and \eqref{BND:PHING}.\\[1ex] Now, let $\eta\in C^\infty_c(\Omega;[0,1])$ be
	any function that is not identically zero. Then equation
	\eqref{EQ:WSD3} implies for all such $\eta$ that
	\begin{align*}
	\int\limits_\Omega \grad\phi^{n+1} \cdot \grad \eta \dx 
	+ \int\limits_\Omega F'(\phi^{n+1}) \eta \dx 
	= \int\limits_\Omega\mu^{n+1} \eta\dx. 
	\end{align*}
	Due to (P2.3) with $\delta=1$, the second summand on the left-hand side can be bounded by
	\begin{align}
	\label{EST:DF}
	\left| \int\limits_\Omega F'(\phi^{n+1}) \eta \dx \right| 
	&\le \int\limits_\Omega F_1(\phi^{n+1}) \dx + B_F|\Omega|^{1/2}\|\phi^{n+1}\|_{L^2(\Omega)} + \big(A^1_F + B_F\big)|\Omega| \notag\\[1ex]
	&\le E(\phi_0) +C^* + B_F|\Omega|^{1/2}\|\phi^{n+1}\|_{L^2(\Omega)} + \big(A^1_F + B_F\big)|\Omega|  \;\;\le C.
	\end{align}
	Hence there exists some constant $C(\eta)\ge 0$ such that
	\begin{align}
	\label{EST:MU}
	\left| \int\limits_\Omega \mu^{n+1}\,\eta \dx \right| \le C(\eta).
	\end{align}
	Let us now define the set
	\begin{align*}
	\mathcal M_\eta:=\left\{ v\in H^1(\Omega) \;\Big\vert\; \left|\textstyle\int_\Omega v \eta \dx \right| \le C(\eta) \right\}.
	\end{align*}
	This set is obviously a non-empty, closed and convex subset of $H^1(\Omega)$. This means that the generalized Poincar\'e inequality (Lemma \ref{LEM:GPI}) can be applied to this set with $u_0=0$ and 
	\begin{align*}
	C_0:= \frac{C(\eta)}{\left|\textstyle\int_\Omega \eta \dx\right|}
	\end{align*}
	because then all numbers $\xi\in\R$ with $\xi\chi_\Omega\in\mathcal M_\eta$ satisfy
	\begin{align*}
	|\xi| \le \frac{\left|\textstyle\int_\Omega \xi\eta \dx\right|}{\left|\textstyle\int_\Omega \eta \dx\right|} \le C_0.
	\end{align*}
	We obtain that 
	\begin{align*}
	\|\mu\|_{L^2(\Omega)} \le C\big(1+\|\grad\mu\|_{L^2(\Omega)}
	\big),\quad \mbox{ for all } \mu\in\mathcal M_\eta.
	\end{align*}
	We know from estimate \eqref{EST:MU} that $\mu^{n+1}\in\mathcal M_\eta$ and thus 
	\begin{align}
	\label{EST:MU2}
	\|\mu^{n+1}\|_{L^2(\Omega)} \le C\big(1+\|\grad\mu^{n+1}\|_{L^2(\Omega)} \big),\quad \text{for all}\;n\in\{0,...,N-1\}.
	\end{align}
	To establish a uniform bound on $\mu_N$, let $n\in\{1,...,N\}$ be arbitrary and set $t:=n\tau$. Then, for any $s\in]t-\tau,t]$, we have
	\begin{align*}
	\phi_N(s) = \phi_N(t) = \phi^n_N,\quad \mu_N(s)=\mu_N(t)=\mu^n_N \tand \mu_{\Gamma,N}(s)=\mu_{\Gamma,N}(t)=\mu_{\Gamma,N}^n.
	\end{align*}
	Using the a priori estimate \eqref{IEQ:APB} and the definition of $\mu_N$ and $\mu_{\Gamma,N}$, we obtain that
	\begin{align*}
	&E\big(\phi_N(t)\big) + \frac 1 2 \int\limits_{t-\tau}^t \|\grad\mu_N(s)\|^2_{L^2(\Omega)} 
	+ \|\gradg\mu_{\Gamma,N}(s)\|^2_{L^2(\Gamma)} \ds \\
	&\quad = E\big(\phi_N(t)\big) + \frac 1 2 \int\limits_{t-\tau}^t \frac{1}{\tau^2}\|\phi_N(s)-\phi_N(s-\tau)\|^2_\Va \ds \\
	&\quad = E\big(\phi_N(t)\big) + \frac 1 2 \int\limits_{t-\tau}^t \frac{1}{\tau^2}\|\phi_N(t)-\phi_N(t-\tau)\|^2_\Va \ds\\
	&\quad = E\big(\phi_N(t)\big) + \frac 1 {2\tau} \|\phi_N(t)-\phi_N(t-\tau)\|^2_\Va\\[2mm]
	&\quad \le E\big(\phi_N(t-\tau)\big).
	\end{align*}
	It follows inductively that
	\begin{align}
	\label{EST:NRG}
	E\big(\phi_N(t)\big) + \frac 1 2 \int\limits_{0}^t \|\grad\mu_N(s)\|^2_{L^2(\Omega)} 
	+ \|\gradg\mu_{\Gamma,N}(s)\|^2_{L^2(\Gamma)} \ds \le E(\phi_0)
	\end{align}
	and in particular,
	\begin{align}
	\label{EST:GMU}
	\int\limits_0^T  \|\grad \mu_N(s)\|_{L^2(\Omega)}^2 + \|\grad \mu_{\Gamma,N}(s)\|_{L^2(\Gamma)}^2 \ds \le 2\big(E(\phi_0) + C^*\big)
	\le C.
	\end{align}
	Inequality \eqref{EST:MU2} directly yields
	\begin{align}
	\|\mu_N(s)\|_{L^2(\Omega)} \le C\big(1+\|\grad\mu_N(s)\|_{L^2(\Omega)} \big), \quad s\in]0,T]
	\end{align}
	and thus we can conclude that $\mu_N$ is uniformly bounded in $L^2\big(0,T;H^1(\Omega)\big)$ which means that \eqref{BND:MUN} is established. \\[1ex]
	Testing equation \eqref{EQ:WSD3} with $\eta\equiv 1$ gives
	\begin{align}
	\label{EQ:MUGN}
	\int\limits_{\Gamma} \mu_{\Gamma,N}(t) \dS
	= \int\limits_{\Gamma}  G'\big(\phi_N(t)\big)\dS
	+ \int\limits_{\Omega} F'\big(\phi_N(t)\big) \dx
	- \int\limits_{\Omega} \mu_N(t) \dx, 
	\end{align}
	for all $t\in[0,T]$. From (P2.3) with $\delta=1$ and estimate \eqref{EST:NRG} we deduce that
	\begin{align}
	\label{EST:DG}
	\Bigg|\,\int\limits_{\Gamma} G'\big(\phi_N(t)\big) \dS\Bigg|
	&\le  \int\limits_\Gamma G_1\big(\phi_N(t)\big) \dx
	+ B_G |\Gamma|^{1/2} \|\phi_N(t)\|_{L^2(\Gamma)}
	+ (A_G^1 + B_G) |\Gamma_T| \notag\\[1ex]
	&\le E(\phi_0) + C^* + B_G |\Gamma|^{1/2} \|\phi_N\|_{L^\infty(0,T;L^2(\Gamma))}
	+ (A_G^1 + B_G) |\Gamma_T|.
	\end{align}
	for all $t\in[0,T]$. Recall that $\phi_N$ is uniformly bounded in $L^\infty\big(0,T;L^2(\Gamma)\big)$. The second integral on the right-hand side of \eqref{EQ:MUGN} can be bounded analogously and we obtain that
	\begin{align*} 
	\Bigg|\int\limits_{\Gamma} G'\big(\phi_N(t)\big) \dS \Bigg| \le C \tand \Bigg|\int\limits_{\Omega} F'\big(\phi_N(t)\big) \dx \Bigg| \le C, \quad t\in[0,T].
	\end{align*}
	Consequently,
	\begin{align*}
	\Bigg|\int\limits_\Gamma \mu_{\Gamma,N}(t) \dS\,\Bigg| \le C + \int\limits_\Omega |\mu_{N}(t)| \dx, \quad t\in[0,T] 
	\end{align*}
	and thus, since $\mu_N$ is uniformly bounded in $L^2(0,T;L^2(\Omega))$,
	\begin{align*}
	\int\limits_0^T \|\langle\mu_{\Gamma,N}(t)\rangle_\Gamma \|_{L^2(\Gamma)}^2 \dt 
	&\le C + C\int\limits_0^T \Bigg(\int\limits_\Omega |\mu_{N}(t)| \dx\Bigg)^2\mathrm dt \\
	&\le C+C\|\mu_N\|_{L^2(0,T;L^2(\Omega))}^2 \le C.
	\end{align*}
	Then, by Poincar\'e's inequality,
	\begin{align*}
	&\int\limits_0^T \|\mu_{\Gamma,N}(t)\|_{L^2(\Gamma)}^2 \dt\\
	&\quad\le C\int\limits_0^T \|\langle\mu_{\Gamma,N}(t)\rangle_\Gamma \|_{L^2(\Gamma)}^2 \dt
	+ C \int\limits_0^T \|\mu_{\Gamma,N}(t) - \langle\mu_{\Gamma,N}(t)\rangle_\Gamma\|_{L^2(\Gamma)}^2 \dt\\
	&\quad\le C + C \int\limits_0^T \|\grad_\Gamma \mu_{\Gamma,N}(t)\|_{L^2(\Gamma)}^2 \dt
	\end{align*}
	for all $t\in[0,T]$. Finally, \eqref{EST:GMU} implies that $\mu_{\Gamma,N}$ is uniformly bounded in the space $L^2\big(0,T;H^1(\Gamma)\big)$ which proves assertion \eqref{BND:MUNG}.
\end{proof}

\subsection{H\"older estimates for the piecewise continuous extension} 
Via interpolation type arguments we can show H\"older continuity in
time.

\begin{lemma}
	\label{LEM:HOE}
	There exists some constant $C>0$ that does not depend on $N$ such that for all $t_1,t_2\in[0,T]$,
	\begin{align}
	\label{ASS:S31}
	\|\bar\phi_N(t_1)-\bar\phi_N(t_2)\|_{L^2(\Omega)} &\le C\,|t_1-t_2|^{1/4},  \\
	\label{ASS:S32}
	\|\bar\phi_N(t_1)-\bar\phi_N(t_2)\|_{L^2(\Gamma)} &\le C\,|t_1-t_2|^{1/4}, \qquad\text{if}\; \kappa>0,\\
	\label{ASS:S33}
	\|\bar\phi_N(t_1)-\bar\phi_N(t_2)\|_{H^1(\Gamma)^*} &\le C\,|t_1-t_2|^{1/4}, \qquad\text{if}\; \kappa=0.	
	\end{align}
\end{lemma}

\begin{proof}
	Let $t_1,t_2\in[0,T]$ be arbitrary. Without loss of generality, we assume that $t_1<t_2$. Since $\bar\phi_N$ is piecewise linear in time, it is weakly differentiable with respect to $t$ and we can rewrite the equations \eqref{EQ:WSD1} and \eqref{EQ:WSD2} as 
	\begin{align}
	\label{EQ:BPHI}
	&\int\limits_\Omega \partial_t \bar\phi_N(t)\, \zeta \dx 
	= -\int\limits_\Omega \grad\mu_N(t) \cdot \grad\zeta \dx
	&& \hspace{-15mm}\text{for all}\; \zeta\in H^1(\Omega),\\
	\label{EQ:BPHIG}
	&\int\limits_\Gamma \partial_t \bar\phi_N(t)\, \xi \dS 
	= -\int\limits_\Gamma \gradg \mu_{\Gamma,N}(t) \cdot \gradg \xi \dS
	&& \hspace{-15mm}\text{for all}\; \xi\in H^1(\Gamma)
	\end{align}
	for all $t\in[0,T]$. Choosing $\zeta:=\bar\phi_N(t_2)-\bar\phi_N(t_1)$
	and integrating with respect to $t$ from $t_1$ to $t_2$ we obtain that\vspace{-3mm}
	\begin{align*}
	&\|\bar\phi_N(t_2)-\bar\phi_N(t_1)\|_{L^2(\Omega)}^2
	= -\int\limits_{t_1}^{t_2}\int\limits_\Omega \grad\mu_N(t) \cdot \big(\grad\bar\phi_N(t_2)-\grad\bar\phi_N(t_1)\big) \dx\dt \\[1mm]
	&\quad \le 2\|\phi_N\|_{L^\infty(0,T;H^1(\Omega))}\, \|\mu_N\|_{L^2(0,T;H^1(\Omega))}\, |t_2-t_1|^{1/2} \le C\, |t_2-t_1|^{1/2},
	\end{align*}
	which proves assertion \eqref{ASS:S31}. Assertion \eqref{ASS:S32} can be proved analogously by choosing the test function ${\xi:=\bar\phi_N\vert_\Gamma(t_2)-\bar\phi_N\vert_\Gamma(t_1)}$. However, if $\kappa=0$, we have to proceed differently since $(\bar\phi_N\vert_\Gamma)_{N\in\N}$ is merely bounded in $L^\infty\big(0,T;L^2(\Gamma)\big)$. As $\langle \bar\phi(t_2)-\bar\phi(t_1) \rangle_\Gamma = 0$ it holds that
	\begin{align*}
	&\|\bar\phi(t_2)-\bar\phi(t_1)\|^2_{H^1(\Gamma)^*} \\
	&\quad= \int\limits_\Gamma  
	\gradg(-\laplace_\Gamma)^{-1} \big(\bar\phi_N(t_2)-\bar\phi_N(t_1)\big) \cdot 
	\gradg(-\laplace_\Gamma)^{-1} \big(\bar\phi_N(t_2)-\bar\phi_N(t_1)\big) \dS\\
	&\quad= \int\limits_\Gamma  
	\big(\bar\phi_N(t_2)-\bar\phi_N(t_1)\big) \;
	(-\laplace_\Gamma)^{-1} \big(\bar\phi_N(t_2)-\bar\phi_N(t_1)\big) \dS.	
	\end{align*}
	Hence, choosing $\xi:=(-\laplace_\Gamma)^{-1} \big(\bar\phi_N(t_2)-\bar\phi_N(t_1)\big)\in H^1(\Gamma)$ yields
	\begin{align*}
	&\|\bar\phi(t_2)-\bar\phi(t_1)\|^2_{H^1(\Gamma)^*} 
	= \int\limits_\Gamma  
	\big(\bar\phi_N(t_2)-\bar\phi_N(t_1)\big) \; \xi \dS
	=  -\int\limits_{t_1}^{t_2} \int\limits_\Gamma \gradg \mu_{\Gamma,N}(t) \cdot \gradg \xi \dS\dt	\\
	&\quad = \int\limits_{t_1}^{t_2}\int\limits_\Gamma \mu_{\Gamma,N}(t) \;
	\big(\bar\phi_N(t_2)-\bar\phi_N(t_1)\big) \dS \dt\\
	&\quad \le 2 \|\mu_{\Gamma,N}\|_{L^2(0,T;L^2(\Gamma))}\, \|\bar\phi_N\|_{L^\infty(0,T;L^2(\Gamma))}
	\, |t_2-t_1|^{1/2},
	\end{align*}
	which proves assertion \eqref{ASS:S33}.
\end{proof}

\subsection{Convergence of the approximate solution}
In this section we will use the apriori estimates and compactness
arguments to show convergence of the time discrete solutions.

\begin{lemma}
	\label{LEM:CONV}
	There exist functions 
	\begin{align*}
	&\phi\in L^\infty\big(0,T;H^1(\Omega)\big) \cap C^{0,\frac 1 4}\big([0,T];L^2(\Omega)\big)\\
	&\psi \in
	\begin{cases}
	L^\infty\big(0,T;H^1(\Gamma)\big) \cap C^{0,\frac 1 4}\big([0,T];L^2(\Gamma)\big), &\kappa>0,\\
	L^\infty\big(0,T;H^{1/2}(\Gamma)\big) \cap C^{0,\frac 1 4}\big([0,T];H^1(\Gamma)^*\big) \cap C\big([0,T];L^2(\Gamma)\big), &\kappa=0,
	\end{cases}\\[2mm]
	&\mu\in L^2\big(0,T;H^1(\Omega)\big),\\
	&\mu_\Gamma\in L^2\big(0,T;H^1(\Gamma)\big)
	\end{align*}
	such that for any $\gamma\in]0,\frac 1 4[$,
	\begin{align}
	\label{ASS:CPHI4}
	\phi_N &\wsto \phi &&\hspace{-25mm} \quad \text{in}\;\; L^\infty\big(0,T;H^1(\Omega)\big),\\
	\label{ASS:CPHI}
	\bar\phi_N &\to \phi &&\hspace{-25mm} \quad \text{in}\;\; C^{0,\gamma}\big([0,T];L^2(\Omega)\big), \\
	\label{ASS:CPHI2}
	\phi_N &\to \phi &&\hspace{-25mm} \quad \text{in}\;\; L^\infty\big(0,T;L^2(\Omega)\big),\\
	\label{ASS:CPHI3}
	\phi_N &\to \phi &&\hspace{-25mm} \quad \text{almost everywhere in}\; \Omega_T, \\[2mm]
	\label{ASS:CPHIG4}
	\phi_N\vert_\Gamma &\wsto \psi  &&\hspace{-25mm} \quad \text{in}\;\; 
	\begin{cases}
	L^\infty\big(0,T;H^1(\Gamma)\big) , &\kappa>0,\\
	L^\infty\big(0,T;H^{1/2}(\Gamma)\big) , &\kappa=0,
	\end{cases}\\[2mm]
	\label{ASS:CPHIG}
	\bar\phi_N\vert_\Gamma &\to \psi  &&\hspace{-25mm} \quad \text{in}\;\; 
	\begin{cases}
	C^{0,\gamma}\big([0,T];L^2(\Gamma)\big), &\kappa>0,\\
	C^{0,\gamma}\big([0,T];H^1(\Gamma)^*\big), &\kappa=0,
	\end{cases}\\[2mm]
	\label{ASS:CPHIG2}
	\phi_N\vert_\Gamma &\to \psi  &&\hspace{-25mm} \quad \text{in}\;\; 
	\begin{cases}
	L^\infty\big(0,T;L^2(\Gamma)\big) , &\kappa>0,\\
	L^\infty\big(0,T;H^1(\Gamma)^*\big) , &\kappa=0,
	\end{cases}\\[2mm]
	\label{ASS:CPHIG3}
	\phi_N\vert_\Gamma &\to \psi  &&\hspace{-25mm} \quad \text{almost everywhere on}\; \Gamma_T,  \\
	\label{ASS:CPHI5}
	\phi_N\vert_\Gamma &\to \psi &&\hspace{-25mm} \quad \text{in}\;\; C\big([0,T];L^2(\Gamma)\big),
	\\[2mm]
	\label{ASS:CMU}
	\mu_N &\wto \mu &&\hspace{-25mm} \quad \text{in}\;\; L^2\big(0,T;H^1(\Omega)\big), \\
	\label{ASS:CMUG}
	\mu_{\Gamma,N} &\wto \mu_\Gamma &&\hspace{-25mm} \quad \text{in}\;\; L^2\big(0,T;H^1(\Gamma)\big)
	\end{align}
	up to a subsequence as $N\to\infty$. Moreover, it holds that $\psi=\phi\vert_\Gamma$ almost everywhere on $\Gamma_T$.
\end{lemma}

\begin{proof}
	Due to the bounds \eqref{BND:PHIN}-\eqref{BND:MUNG} that have been established in Lemma \ref{LEM:BND}, we can conclude from the Banach-Alaoglu theorem that there exist functions 
	\begin{align*}
	&\phi\in L^\infty\big(0,T;H^1(\Omega)\big),
	&&\psi \in
	\begin{cases}
	L^\infty\big(0,T;H^1(\Gamma)\big), &\kappa>0,\\
	L^\infty\big(0,T;H^{1/2}(\Gamma)\big), &\kappa=0,
	\end{cases}\\[2mm]
	&\mu\in L^2\big(0,T;H^1(\Omega)\big),
	&&\mu_\Gamma\in L^2\big(0,T;H^1(\Gamma)\big)
	\end{align*}
	such that, after extraction of a subsequence, 
	\begin{alignat*}{4}
	&\phi_N \wsto \phi \quad
	&&\text{in}\quad L^\infty\big(0,T;H^1(\Omega)\big),\qquad
	&&\phi_N\vert_\Gamma \wsto \psi \quad
	&&\text{in}\quad 
	\begin{cases}
	L^\infty\big(0,T;H^1(\Gamma)\big), &\kappa>0,\\
	L^\infty\big(0,T;H^{1/2}(\Gamma)\big), &\kappa=0,
	\end{cases}\\[2mm]
	&\mu_N \wto \mu 	\quad 
	&&\text{in}\quad L^2\big(0,T;H^1(\Omega)\big),
	&&\mu_{\Gamma,N} \wto \mu_\Gamma	\quad
	&&\text{in}\quad L^2\big(0,T;H^1(\Gamma)\big).
	\end{alignat*}
	Hence, the assertions \eqref{ASS:CPHI4}, \eqref{ASS:CPHIG4}, \eqref{ASS:CMU} and \eqref{ASS:CMUG} are established. \\[1ex]
	Recall that $\bar\phi_N$ lies in $C\big([0,T];L^2(\Omega)\big)$ and is bounded uniformly in $L^\infty\big(0,T;H^1(\Omega)\big)$ by 
	\begin{align*}
	\|\bar\phi_N\|_{L^\infty(0,T;H^1(\Omega))} \le \|\phi_N\|_{L^\infty(0,T;H^1(\Omega))} \le C_1
	\end{align*} 
	according to \eqref{BND:PHIN} for all $N\in\N$. Since the embedding
	from $H^1(\Omega)$ to $L^2(\Omega)$ is compact, we can use the
	equicontinuity of $(\bar\phi_N)_{N\in\N}$ (which follows directly from
	\eqref{ASS:S31}) to apply the theorem of Arzelà--Ascoli for functions
	with values in a Banach space (see J. Simon
	\cite[Lem.\,1]{simon}). The theorem implies that
	\begin{align*}
	\bar\phi_N \to \phi \quad \text{in}\;\; C([0,T];L^2(\Omega)\big).
	\end{align*}
	Using \eqref{ASS:S31} one can easily show that $\phi\in C^{0,\frac 1 4}\big([0,T];L^2(\Omega)\big)$. For any $\gamma\in]0,\tfrac 1 4[$, we obtain by interpolation that
	\begin{align*}
	\|\cdot\|_{C^{0,\gamma}([0,T];L^2(\Omega))} 
	\le C\, \|\cdot\|_{C^{0,1/4}([0,T];L^2(\Omega))}^{4\gamma} \|\cdot\|_{C([0,T];L^2(\Omega))}^{1-4\gamma}.
	\end{align*}
	Hence, it also holds that
	\begin{align*}
	\bar\phi_N \to \phi \quad \text{in}\;\; C^{0,\gamma}\big([0,T];L^2(\Omega)\big)
	\end{align*}
	for every $\gamma\in]0,\frac 1 4[$. This proves \eqref{ASS:CPHI}. Assertion \eqref{ASS:CPHIG} can be established analogously; in the case $\kappa=0$ we have to use the compact embedding $L^2(\Gamma) \cong L^2(\Gamma)^* \hookrightarrow H^1(\Gamma)^*$ instead.\\[1ex]
	For any $t\in[0,T]$ we can choose $n\in\{1,...,N\}$ and $\alpha\in [0,1]$ such that $t=\alpha n\tau + (1-\alpha)(n-1)\tau$. We have
	\begin{align*}
	&\|\bar\phi_N(t)-\phi_N(t)\|_{L^2(\Omega)} \le \| \alpha \phi^n + (1-\alpha) \phi_N^{n-1}(t) - \phi_N^n(t) \|_{L^2(\Omega)} \\
	&\quad = (1-\alpha)\,\| \phi_N^n(t)  - \phi_N^{n-1}(t) \|_{L^2(\Omega)} \le C\tau^{1/4}
	\end{align*}
	which tends to zero as $N$ tends to infinity. Proceeding similarly, we obtain that
	\begin{align*}
	\|\bar\phi_N(t)-\phi_N(t)\|_{L^2(\Gamma)} &\le C\tau^{1/4}
	\underset{N\to\infty}{\longrightarrow} 0, &&\hspace{-20mm}\text{if}\; \kappa >0,\\
	\|\bar\phi_N(t)-\phi_N(t)\|_{H^1(\Gamma)^*} &\le C\tau^{1/4}
	\underset{N\to\infty}{\longrightarrow} 0, &&\hspace{-20mm}\text{if}\; \kappa =0.
	\end{align*}
	Together with \eqref{ASS:CPHI} and \eqref{ASS:CPHIG} this implies \eqref{ASS:CPHI2} and \eqref{ASS:CPHIG2}. In particular, this means that $\phi_N \to \phi$ in $L^2(\Omega_T)$ and $\phi_N\vert_\Gamma \to \phi\vert_\Gamma$ in $L^2(\Gamma_T)$ if $\kappa>0$. Therefore we can extract a subsequence that converges almost everywhere. This proves \eqref{ASS:CPHI3} and \eqref{ASS:CPHIG3} in the case $\kappa >0$.\\[1ex]
	By integration, \eqref{EQ:BPHIG} implies that
	\begin{align*}
	\int\limits_{\Gamma_T} \delt\bar\phi_N \,\xi \dSx\dt = -\int\limits_{\Gamma_T} \grad_\Gamma \mu_{\Gamma,N} \cdot \grad_\Gamma \xi \dxt
	\end{align*}
	for all functions $\xi\in L^2\big(0,T;H^1(\Gamma)\big)$ with $\delt \xi\in L^2(\Gamma_T)$ and consequently
	\begin{align*}
	\delt\bar\phi_N = \laplace_\Gamma \mu_{\Gamma,N} \quad\text{in}\; L^2\big(0,T;H^1(\Gamma)^*\big)
	\end{align*}
	for all $N\in\N$. This means that $\delt\phi_N$ is uniformly bounded by
	\begin{align*}
	\|\delt\bar\phi_N\|_{L^2(0,T;H^1(\Gamma)^*)} 
	= \|\laplace_\Gamma\mu_{\Gamma,N}\|_{L^2(0,T;H^1(\Gamma)^*)} 
	= \left\|\grad_\Gamma \mu_{\Gamma,N} \right\|_{L^2(0,T;L^2(\Gamma))} \le C_5.
	\end{align*}
	according to \eqref{BND:MUNG} and the definition of $\mu_{\Gamma,N}$. Hence, the sequence $(\bar\phi_N)_{N\in\N}$ is bounded in $L^\infty\big(0,T;H^{1/2}(\Gamma)\big) \cap H^1\big(0,T;H^1(\Gamma)^*\big)$.
	Since $H^{1/2}(\Gamma)$ is compactly embedded in $L^2(\Gamma)$ and $L^2(\Gamma) \cong L^2(\Gamma)^*$ is continuously embedded in $H^1(\Gamma)^*$ we can use the Aubin-Lions lemma (cf. J. Simon\cite[Cor.\,5]{simon}) to conclude that 
	\begin{align*}
	\bar\phi_N\vert_\Gamma \to \psi \;\text{in}\; C\big([0,T];L^2(\Gamma)\big).
	\end{align*}
	up to a subsequence which is \eqref{ASS:CPHI5}. Then assertion \eqref{ASS:CPHIG3} in the case $\kappa =0$ immediately follows after another subsequence extraction. \\[1ex]
	Similarly, we can use \eqref{EQ:BPHI} to conclude that
	\begin{align*}
	\delt\bar\phi_N = \laplace \mu_{\Gamma,N} \quad\text{in}\; L^2\big(0,T;H^1(\Omega)^*\big)
	\end{align*}
	with
	\begin{align*}
	\|\delt\bar\phi_N\|_{L^2(0,T;H^1(\Omega)^*)} 
	= \|\laplace_\Gamma\mu_{\Gamma,N}\|_{L^2(0,T;H^1(\Omega)^*)} 
	= \left\|\grad_\Gamma \mu_{\Gamma,N} \right\|_{L^2(0,T;L^2(\Omega))} \le C_4
	\end{align*}
	and therefore, $(\bar\phi_N)_{N\in\N}$ is bounded in $L^\infty\big(0,T;H^1(\Omega)\big) \cap H^1\big(0,T;H^1(\Omega)^*\big)$. It holds that
	\begin{align*}
	H^1(\Omega) \hookrightarrow H^{s}(\Omega) \hookrightarrow L^2(\Omega) \cong L^2(\Omega)^* \hookrightarrow H^1(\Omega)^*\quad\text{for any}\; s\in\,]0,1[
	\end{align*}
	where all embeddings are continuous and at least the first embedding is compact. Now, the Aubin-Lions lemma (cf. J. Simon\cite[Cor.\,5]{simon}) implies that $\phi_N\vert_\Gamma \to \phi$ in $C\big([0,T];H^s(\Omega)\big)$ for any fixed $s\in\,]0,1[$ up to a subsequence. Choosing $s>\frac 1 2$, we can conclude from the continuity of the trace operator (see (P5)) that
	\begin{align*}
	\bar\phi_N\vert_\Gamma \to \phi\vert_\Gamma \;\text{in}\; C\big([0,T];L^2(\Gamma)\big).
	\end{align*}
	after subsequence extraction. This finally proves $\phi\vert_\Gamma = \psi$.
\end{proof}

\subsection{Proof of the existence theorem} 

We can finally prove our main result Theorem \ref{MAIN} by showing that the limit $(\phi,\mu,\mu_\Gamma)$ from Lemma \ref{LEM:CONV} is a weak solution of the Cahn--Hilliard equation \eqref{CHLW} in the sense of Definition \ref{DEF:WS}.\\[1ex]

\textit{Proof of Theorem \ref{MAIN}.}\hspace{4pt}
First note that $(\phi,\mu,\mu_\Gamma)$ has the desired regularity according to Lemma \ref{LEM:CONV}. This means that item (i) of Definition \ref{DEF:WS} is already satisfied. To verify (ii), we need to pass the limit in the discretized system \eqref{EQ:WSD1}-\eqref{EQ:WSD3} using the convergence results that were established in Lemma \ref{LEM:CONV}. First, the equations \eqref{EQ:BPHI} and \eqref{EQ:BPHIG} imply that
\begin{align*}
\int\limits_{\Omega_T} (\bar\phi_N - \phi_0)\delt\zeta \dxt 
= \int\limits_{\Omega_T} \grad\mu_N\cdot \grad\zeta \dxt 
\end{align*}
for all $\zeta\in L^2\big(0,T;H^1(\Omega)\big)$ with $\delt\zeta\in L^2(\Omega_T)$ and $\zeta(T)=0$ and
\begin{align*}
\int\limits_{\Gamma_T} (\bar\phi_N - \phi_0)\delt\xi\dSx\dt
= \int\limits_{\Gamma_T} \grad_\Gamma \mu_{\Gamma,N}\cdot \grad\xi \dSx \dt
\end{align*}
for all $\xi\in L^2\big(0,T;H^1(\Gamma)\big)$ with $\delt\xi\in L^2(\Gamma_T)$ and $\xi(T)=0$. Using the convergence properties from Lemma~\ref{LEM:CONV}, we obtain the weak formulations \eqref{DEF:WF1} and \eqref{DEF:WF2} from Definition \ref{DEF:WS}. 
Now, let $\eta\in L^2\big(0,T;\V\big) \cap L^\infty(\Omega_T)$ be arbitrary. Testing \eqref{EQ:WSD3} with $\eta$ and integrating with respect to $t$ from $0$ to $T$ we obtain that
\begin{align*}
\int\limits_{\Omega_T} \grad\phi_N\cdot\grad\eta + F'(\phi_N)\,\eta\dxt 
+\int\limits_{\Gamma_T} \kappa \gradg \phi_N \cdot \gradg \eta + G'(\phi_N) \eta \dS \mathrm dt \notag \\
= \int\limits_{\Omega_T} \mu_N\,\eta\dxt + \int\limits_{\Gamma_T} \mu_{\Gamma,N}\,\eta \dS\mathrm dt.
\end{align*}
One can easily see that the terms that depend linearly on $\phi_N$,
$\mu_N$ or $\mu_{\Gamma,N}$ are converging to the corresponding terms
in \eqref{DEF:WF3} due to the convergence results that were
established in Lemma \ref{LEM:CONV}. Recall that
$G'(\phi_N)\in L^1(\Gamma_T)$ for all $N\in\N$ and
$G'(\phi_N) \to G'(\phi)$ as $N\to\infty$ almost everywhere in
$\Gamma_T$. Now, let $\varepsilon>0$ be arbitrary. For any $\alpha>0$,
let $\Gamma_\alpha$ denote a measurable subset of $\Gamma_T$ with
$|\Gamma_\alpha|<\alpha$. Using (P2) and the estimate
$\int_\Gamma G\big(\phi_N(t)\big)\mathrm dS \le E(\phi_0)+C^*$ we
obtain that
\begin{align*}
&\int\limits_{\Gamma_\alpha}|G'(\phi_N)|\dS\mathrm dt \\
&\quad\le \delta T \big(E(\phi_0)+C^*\big) + (A_G^\delta + B_G) |\Gamma_\alpha| + B_G\, T^{1/2}|\Gamma_\alpha|^{1/2}\|\phi_N\|_{L^\infty(0,T;L^2(\Gamma))}
\end{align*}
for any $\delta>0$ where $C^*$ is the constant from \eqref{EST:PHI}. Fixing $\delta:= \eps / \big[2T\big(E(\phi_0) + C^*\big)\big]$ and assuming that $\alpha$ is sufficiently small, we obtain that
\begin{align*}
\int\limits_{\Gamma_\alpha}|G'(\phi_N)|\dS\mathrm dt \le \frac \eps 2 + (A_G^\delta + B_G) \alpha + B_G\, T^{1/2}\alpha^{1/2}C_1 < \eps.
\end{align*}
Therefore we can apply Vitali's convergence theorem (cf. \cite[p.\,57]{alt}) which implies that $G'(\phi)\in L^1(\Gamma_T)$ with
$G'(\phi_N)\to G'(\phi)$ in $L^1(\Gamma_T)$ and thus
\begin{align*}
\int\limits_{\Gamma_T} G'(\phi_N) \eta \dS\mathrm dt \to \int\limits_{\Gamma_T} G'(\phi) \eta\dS\mathrm dt,\quad N\to\infty
\end{align*}
since $\eta\in L^\infty(\Gamma)$. The proof for $F'(\phi)\in L^1(\Omega_T)$ and $\int_\Omega F'(\phi_N) \eta \dx \to \int_\Omega F'(\phi) \eta \dx$ proceeds analogously and we can finally conclude that the triple $(\phi,\mu,\mu_\Gamma)$ satisfies the weak formulations \eqref{DEF:WF1}-\eqref{DEF:WF3} of Definition~\ref{DEF:WS}. This means that we have verified item (ii) of the definition.\\[1ex] 
Let now $t\in [0,T]$ be arbitrary. Then we have $\phi_N(t)\to \phi(t)$ almost everywhere on $\Omega$ and $\phi_N\vert_\Gamma(t)\to \phi\vert_\Gamma(t)$ almost everywhere on $\Gamma$. Moreover, recall that
\begin{align*}
|F_2(\phi_N)| \le C_F\big(1+|\phi_N|^2\big) \quad\text{in}\;\Omega
\tand
|G_2(\phi_N)| \le C_G\big(1+|\phi_N|^2\big) \quad\text{on}\;\Gamma
\end{align*}
due to assumption (P2.4). Hence Lebesgue's general convergence theorem (cf. \cite[p.\,60]{alt}) implies that
\begin{align*}
\int\limits_\Omega F_2\big(\phi_N(t)\big) \dx 
\to \int\limits_\Omega F_2\big(\phi(t)\big) \dx 
\tand
\int\limits_\Gamma G_2\big(\phi_N(t)\big) \dSx 
\to \int\limits_\Gamma G_2\big(\phi(t)\big) \dSx
\end{align*}
as $N\to\infty$. As the remaining terms of the energy $E$ are convex, we can conclude that
\begin{align*}
&E\big(\phi(t)\big) + \frac 1 2 \int\limits_0^t \|\grad\mu(s)\|^2_{L^2(\Omega)} 
+ \|\gradg\mu_{\Gamma}(s)\|^2_{L^2(\Gamma)} \ds \\
&\quad \le \underset{N\to\infty}{\lim\inf} \left\{ E\big(\phi_N(t)\big) + \frac 1 2 \int\limits_{t-\tau}^t \|\grad\mu_N(s)\|^2_{L^2(\Omega)} 
+ \|\gradg\mu_{\Gamma,N}(s)\|^2_{L^2(\Gamma)} \ds\right\} \\[1mm]
&\quad\le E(\phi_0).
\end{align*}
This verifies item (iii) of Definition \ref{DEF:WS} and completes the proof of Theorem \ref{MAIN}.\hfill $\Box$

%
%

\section{Uniqueness of the weak solution}

If the functions $F_2$ and $G_2$ that were introduced in (P2) are additionally Lipschitz continuous, we can even show that the weak solution predicted by Theorem \ref{MAIN} is unique:

\begin{theorem}
	\label{MAIN2}
	We assume that $\Omega\subset\R^d$ (with $d\in\{2,3\}$) is a bounded domain with Lipschitz boundary and that the potentials $F$ and $G$ satisfy the condition \textnormal{(P2)}. Let $T>0$ and $m=(m_1,m_2)\in \R^2$ be arbitrary and let $\phi_0 \in \Vm$ be any initial datum. Moreover, we assume that $F_2$ and $G_2$ are Lipschitz continuous.
	Then the weak solution $(\phi,\mu,\mu_\Gamma)$ that is predicted by Theorem \ref{MAIN} is the unique weak solution of the system \eqref{CHLW}.
\end{theorem}

\begin{proof}
	Let $(\phi,\mu,\mu_\Gamma)$ denote the solution that is given by Theorem \ref{MAIN}. We assume that there is another weak solution $(\tilde\phi,\tilde\mu,\tilde\mu_\Gamma)$ of \eqref{CHLW} in the sense of Definition \ref{DEF:WS} and we consider the difference
	\begin{align*}
	(\bar\phi,\bar\mu,\bar\mu_{\Gamma}):=(\phi - \tilde\phi,\,\mu - \tilde\mu,\,\mu_\Gamma - \tilde\mu_\Gamma).
	\end{align*}
	Due to \eqref{DEF:WF1} it holds that
	\begin{align*}
	\int\limits_{\Omega_T} \bar\phi\, \delt\zeta \dxt =	\int\limits_{\Omega_T} \grad\bar\mu \cdot \grad\zeta \dxt
	\end{align*}
	for all $\zeta\in L^2\big(0,T;H^1(\Omega)\big)$ with $\delt\zeta\in L^2(\Omega_T)$ and $\zeta(T)=0$. For any $t_0\in [0,T]$ and any  $\eta\in L^2\big(0,T;\V\big)\cap L^\infty(\Omega_T)$ with $\eta\vert_\Gamma \in L^\infty(\Gamma_T)$,  we define
	\begin{align*}
	\zeta(\cdot,t):=
	\begin{cases}
	\int_t^{t_0} \eta(\cdot,s) \ds &\text{if}\; t\le t_0,\\
	0 &\text{if}\; t>t_0.
	\end{cases}
	\end{align*}
	Then, since $\zeta$ is a suitable test function, we have
	\begin{align*}
	-\int\limits_{\Omega_{t_0}} \bar\phi\, \eta \dxt 
	= \int\limits_{\Omega_{t_0}} \grad\bar\mu \cdot \grad\left(\int\limits_t^{t_0} \eta\ds\right) \dxt 
	= \int\limits_{\Omega_{t_0}} \grad \left(\int\limits_0^t \bar\mu \ds\right) \cdot \grad \eta \dxt.
	\end{align*}
	This means that 
	\begin{align*}
	\nmo\bar\phi = -\int\limits_0^t \bar\mu \ds + c \quad\tand\quad \delt\nmo\bar\phi = -\bar\mu.
	\end{align*}
	for some constant $c\in\R$. Now, choosing $\eta=\bar\mu$ yields
	\begin{align*}
	\int\limits_{\Omega_{t_0}} \bar\phi\, \bar\mu \dxt 
	&= -\int\limits_{\Omega_{t_0}} \grad \nmo\bar\phi \cdot \grad \delt\nmo\bar\phi \dxt\\
	&= -\frac 1 2 \int\limits_{\Omega_{t_0}} \ddt \Big( \grad \nmo\bar\phi \cdot \grad \nmo\bar\phi \Big) \dxt \\
	&= -\frac 1 2 \int\limits_{\Omega}  \grad \nmo\bar\phi(t_0) \cdot \grad \nmo\bar\phi(t_0)  \dx
	\end{align*}
	since $\bar\phi(0)=0$ and thus $\nmo\bar\phi(0)=0$. In a similar fashion, we can conclude that
	\begin{align*}
	\int\limits_{\Gamma_{t_0}} \bar\phi\, \bar\mu_\Gamma \dxt = -\frac 1 2 \int\limits_{\Gamma}  \gradg \nmg\bar\phi(t_0) \cdot \gradg \nmg\bar\phi(t_0)  \dx.
	\end{align*}
	Summing up both equations gives
	\begin{align}
	\label{EQ:UN1}
	\frac 1 2 \|\bar\phi(t_0)\|_\Va^2 = -\int\limits_{\Omega_{t_0}} \bar\phi\,\bar\mu \dxt - \int\limits_{\Gamma_{t_0}} \bar\phi\,\bar\mu_\Gamma \dS\dt.
	\end{align}
	From the weak formulation \eqref{DEF:WF3} we obtain that
	\begin{align}
	\label{EQ:WFD}
	\int\limits_{\Omega_T} \bar\mu\,\eta\dxt + \int\limits_{\Gamma_T} \bar\mu_\Gamma\,\eta \dS\mathrm dt
	=\int\limits_{\Omega_T} \grad\bar\phi\cdot\grad\eta \dxt +\int\limits_{\Gamma_T} \kappa \gradg \bar\phi \cdot \gradg \eta \dS\mathrm dt\notag\\
	+ \int\limits_{\Omega_T} \big(F'(\phi) - F'(\tilde\phi)\big)\,\eta\dxt 
	+\int\limits_{\Gamma_T} \big(G'(\phi) - G'(\tilde\phi)\big) \eta \dS\mathrm dt
	\end{align}
	for all  $\eta\in L^2\big(0,T;\V\big)\cap L^\infty(\Omega_T)$ with $\eta\vert_\Gamma \in L^\infty(\Gamma_T)$.  Now, for any $M>0$, we choose the test function 
	\begin{align*}
	\eta:=\chi_{[0,t_0]} \mathcal P_M(\bar\phi)
	\end{align*}
	where $\mathcal P_M$ describes a projection of $\R$ onto the interval $[-M,M]$ given by
	\begin{align*}
	\mathcal P_M:\R\to\R,\quad s\mapsto
	\begin{cases}
	s &\text{if}\; |s|<M\\
	\frac{s}{|s|}M &\text{if}\; |s|\ge M\\
	\end{cases},
	\qquad M>0.
	\end{align*}
	Recall that, according to (P2), $F=F_1+F_2$ and $G=G_1+G_2$ where $F_1$ and $G_1$ are convex. Hence, their derivatives $F'$ and $G'$ are monotonically increasing. In particular,
	\begin{align*}
	\big(F_1'(\phi)-F_1'(\tilde\phi)\big) \mathcal P_M(\bar\phi) &\ge 0 \quad \text{a.e. on}\;\Omega_T,\\
	\big(G_1'(\phi)-G_1'(\tilde\phi)\big) \mathcal P_M(\bar\phi) &\ge 0 \quad \text{a.e. on}\;\Gamma_T.
	\end{align*}
	Plugging these estimates into \eqref{EQ:WFD} with $\eta=\chi_{[0,t_0]} \mathcal P_M(\bar\phi)$ gives
	\begin{align}
	\label{EQ:WFD2}
	&\int\limits_{\Omega_{t_0}} \bar\mu\,\mathcal P_M(\bar\phi) \dxt 
	+ \int\limits_{\Gamma_{t_0}} \bar\mu_\Gamma\,\mathcal P_M(\bar\phi) \dS\mathrm dt \notag	\\
	&\quad \ge \int\limits_{\Omega_{t_0}} \grad\bar\phi\cdot\grad\mathcal P_M(\bar\phi) \dxt
	+ \int\limits_{\Gamma_{t_0}} \kappa \gradg \bar\phi \cdot \gradg \mathcal P_M(\bar\phi) \dS\mathrm dt \notag
	\\
	&\qquad + \int\limits_{\Omega_{t_0}} \big(F_2'(\phi) - F_2'(\tilde\phi)\big)\,\mathcal P_M(\bar\phi)\dxt + \int\limits_{\Gamma_{t_0}} \big(G_2'(\phi) - G_2'(\tilde\phi)\big) \mathcal P_M(\bar\phi) \dS\mathrm dt.
	\end{align}
	In the limit $M\to\infty$ we obtain \eqref{EQ:WFD2} with $\chi_{[0,t_0]} \mathcal P_M(\bar\phi)$ replaced by $\bar\phi$. Together with equation \eqref{EQ:UN1} this yields
	\begin{align}
	\label{EQ:UN3}
	&\frac 1 2 \|\bar\phi(t_0)\|_\Va^2 + \|\grad\bar\phi\|_{L^2(\Omega_{t_0})}^2 + \kappa \|\gradg\bar\phi\|_{L^2(\Gamma_{t_0})}^2 \notag\\
	&\quad \le - \int\limits_{\Omega_{t_0}} \big(F'(\phi) - F'(\tilde\phi)\big)\,\bar\phi\dxt - \int\limits_{\Gamma_{t_0}} \big(G'(\phi) - G'(\tilde\phi)\big) \bar\phi \dS\mathrm dt \notag\\
	&\quad \le L_F \|\bar\phi\|_{L^2(\Omega_{t_0})}^2 + L_G \|\bar\phi\|_{L^2(\Gamma_{t_0})}^2
	\end{align}
	where $L_F,L_G\ge 0$ are the Lipschitz costants of $F_2'$ and $G_2'$. Using integration by parts and Young's inequality with $0<\delta\le 1$ we obtain that
	\begin{align}
	\label{EST:BPHI}
	\|\bar\phi\|_{L^2(\Omega)}^2 = \int\limits_\Omega \grad\nmo\bar\phi\cdot\grad\bar\phi\dx \le  \frac{1}{4\delta}\|\bar\phi\|_\Va^2 + \delta \|\grad\bar\phi\|_{L^2(\Omega)}^2.
	\end{align}
	We can bound $\|\bar\phi\|_{L^2(\Gamma)}^2$ using an interpolation inequality (see \cite[Thm.\,II.4.1]{galdi}). In combination with Poincar\'e's inequality and Young's inequality we infer that
	\begin{align*}
	\|\bar\phi\|_{L^2(\Gamma)}^2 
	\le c\|\bar\phi\|_{L^2(\Omega)}^2 + c\|\bar\phi\|_{L^2(\Omega)}\|\bar\phi\|_{H^1(\Omega)} 
	\le \frac{c}{\delta}\|\bar\phi\|_{L^2(\Omega)}^2 + c\delta\|\grad\bar\phi\|_{L^2(\Omega)}^2
	\end{align*}
	for any $0<\delta\le 1$ and some generic constant $c \ge 0$ independent of $\delta$. Now, we apply \eqref{EST:BPHI} with $\delta^2$ instead of $\delta$ to obtain that
	\begin{align*}
	\|\bar\phi\|_{L^2(\Gamma)}^2 
	\le  \frac{\bar c}{\delta^3}\|\bar\phi\|_\Va^2 + \bar c \delta \|\grad\bar\phi\|_{L^2(\Omega)}^2  
	\end{align*}
	for any $0<\delta\le 1$ and a constant $\bar c\ge 0$ independent of $\delta$. Then, inequality \eqref{EQ:UN3} implies that
	\begin{align*}
	&\frac 1 2 \|\bar\phi\|_\Va^2 + \big(1-(L_F+\bar c L_G)\delta \big)\|\grad\bar\phi\|_{L^2(\Omega_{t_0})}^2 + \kappa \|\gradg\bar\phi\|_{L^2(\Gamma_{t_0})}^2 \\
	&\quad\le \frac{\frac 1 4 L_F + \bar c L_G}{\delta^3} \int\limits_0^{t_0}\|\bar\phi\|_\Va^2 \dt.
	\end{align*}
	Now, choosing $\delta:=1/(2L_F+2\bar c L_G)$ yields
	\begin{align*}
	\|\bar\phi(t_0)\|_\Va^2 
	\le C \int\limits_0^{t_0} \|\bar\phi(t)\|_\Va^2
	\end{align*}
	for some constant $C>0$ that depends only on $L_F$, $L_G$ and $\bar c$. 
	Now, as $t_0\in[0,T]$ was arbitrary, Gronwall's lemma implies that $$\|\bar\phi(t)\|_\Va^2 = 0$$ for almost all $t\in[0,T]$. Thus, $\phi=\tilde\phi$ almost everywhere on $\Omega_T$ and $\phi\vert_\Gamma=\tilde\phi\vert_\Gamma$ almost everywhere on $\Gamma_T$. Recall that both $(\phi,\mu,\mu_\Gamma)$ and $(\tilde\phi,\tilde\mu,\tilde\mu_\Gamma)$ satisfy the weak formulation \eqref{DEF:WF3}. Choosing $\eta\in C_c^\infty(\Omega)$ we obtain that
	\begin{align}
	\label{EQ:UN4}
	\int\limits_{\Omega_T} (\mu -  \tilde \mu) \eta \dx 
	= \int\limits_{\Omega_T} \grad\phi\cdot\grad\eta + F'(\phi)\eta \dx - \int\limits_{\Omega_T} \grad\tilde\phi\cdot\grad\eta + F'(\tilde\phi)\eta \dx = 0,
	\end{align}
	since $\phi=\tilde\phi$ almost everywhere in $\Omega_T$. Thus, $\mu=\tilde\mu$ almost everywhere in $\Omega$ directly follows. Testing \eqref{DEF:WF3} with $\eta\in C^\infty(\Omega)$ and using \eqref{EQ:UN4} gives
	\begin{align}
	\begin{aligned}
	&\int\limits_{\Gamma_T} \mu_\Gamma\, \eta \dS - \int\limits_{\Gamma_T} \tilde \mu_\Gamma\, \eta \dS \notag \\
	&\quad= \int\limits_{\Gamma_T} \kappa \gradg\phi\cdot\gradg\eta + G'(\phi)\eta \dx - \int\limits_{\Gamma_T} \kappa \gradg\tilde\phi\cdot\gradg\eta + G'(\tilde\phi)\eta \dx = 0.
	\end{aligned}
	\end{align}
	This implies that $\mu_\Gamma = \tilde\mu_\Gamma$ almost everywhere on $\Gamma$ and completes the proof of Theorem~\ref{MAIN2}.
\end{proof}

\begin{remark}
	Since the unique weak solution from Theorem \ref{MAIN2} exists on the interval $[0,T]$ for any given $T>0$, it can be considered as the unique weak solution on $[0,\infty[$. This means that this solution is global in time.
\end{remark}

\section{Numerical results}

In this section we present several plots of two numerical simulations (see \autoref{fig1} and \autoref{fig2}) for the system \eqref{CHLW:EPS} in two dimensions with $\eps=0.02$. These numerical solutions were computed by a finite element method implemented by D. Trautwein \cite{trautwein} using MATLAB. In both simulations, the domain $\Omega$ is a square that is discretized by a standard Friedrichs-Keller triangulation with step size $h$. Therefore we write $\Omega_h$ and $\Gamma_h$ to denote the discretizations of $\Omega$ and $\Gamma$. The time evolution is computed by an implicit Euler method with a step size of $\eps^3=8\cdot 10^{-6}$. \\[1ex]
In the first simulation, the parameter $\kappa$ is set to $0.02$ and the domain $\Omega$ is the unit square. The spatial step size for the Friedrichs-Keller triangulation is $h=0.01$ which leads to a grid of $101\times 101$ sampling points. \autoref{fig1} visualizes the time evolution of an initial datum $\phi_0$ that is set to zero in every node of $\Omega_h$ but set to one in every node of $\Gamma_h$. The plots show the solution after $5$, $15$, $80$, $200$, $500$ and $2000$ time steps. Note that, as the mass on the boundary is conserved, the solution will remain identically one on the boundary for all time. The phase separation in the bulk leads to a wave-like structure and eventually, as a long-time effect, the pure phase associated with the value $-1$ forms a single circle around the center of the square $\Omega$.\\[1ex]
In the second simulation, we use the parameter $\kappa=0.075$ and the domain $\Omega=]0,0.5[^2$. The spatial step size is $h=0.005$ and therefore the grid also consists of $101\times 101$ sampling points. Now, the initial datum $\phi_0$ attains random values between $-0.1$ and $0.1$ in every grid point of $\Omega_h$ and also random values between $0.4$ and $0.6$ in every node of $\Gamma_h$. \autoref{fig2} shows the corresponding solution after $5$, $15$, $50$, $125$, $300$ and $2500$ time steps. 

\begin{figure}[h!]
	\centering
	\includegraphics[width=0.32\textwidth]{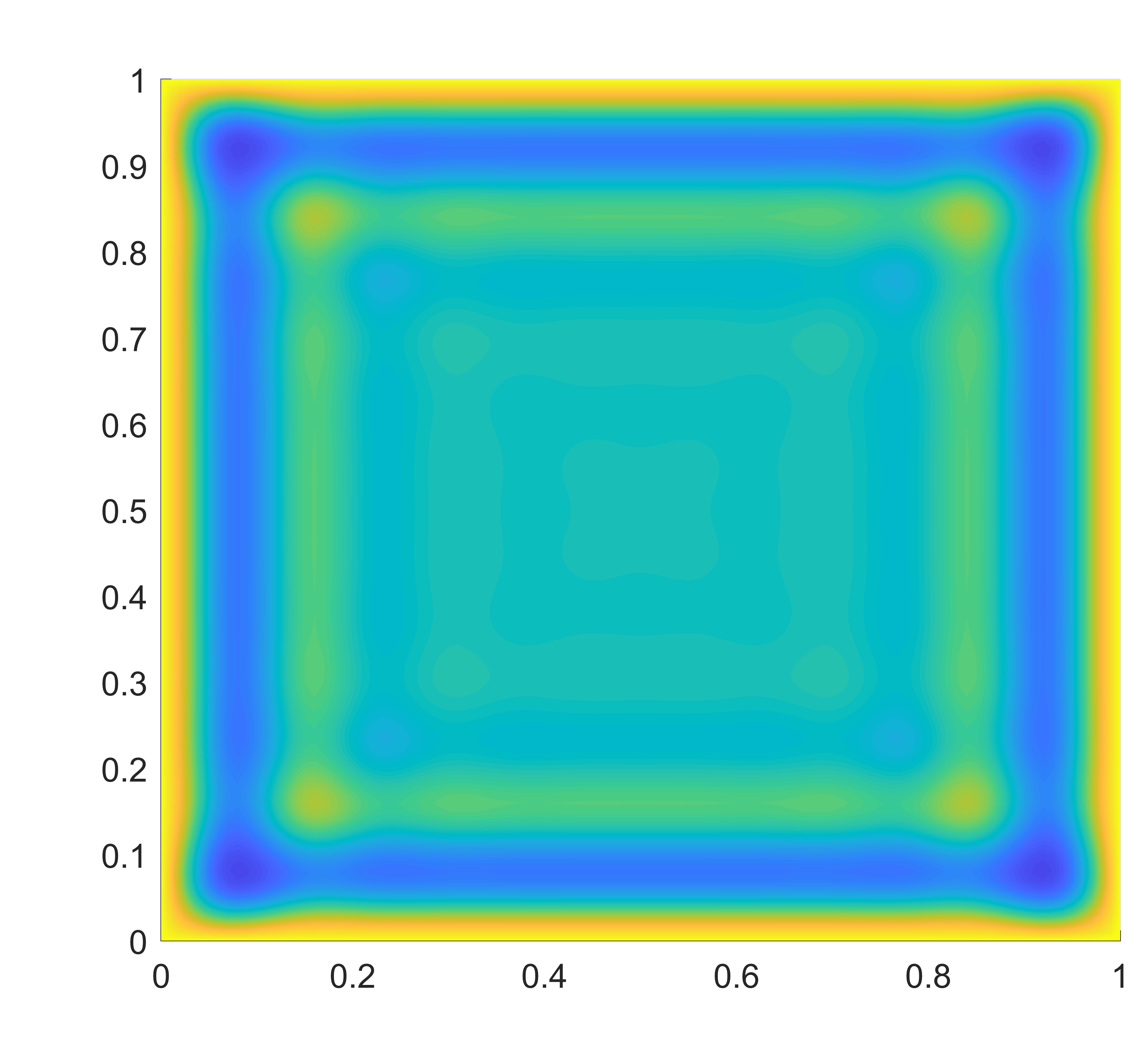}
	\includegraphics[width=0.32\textwidth]{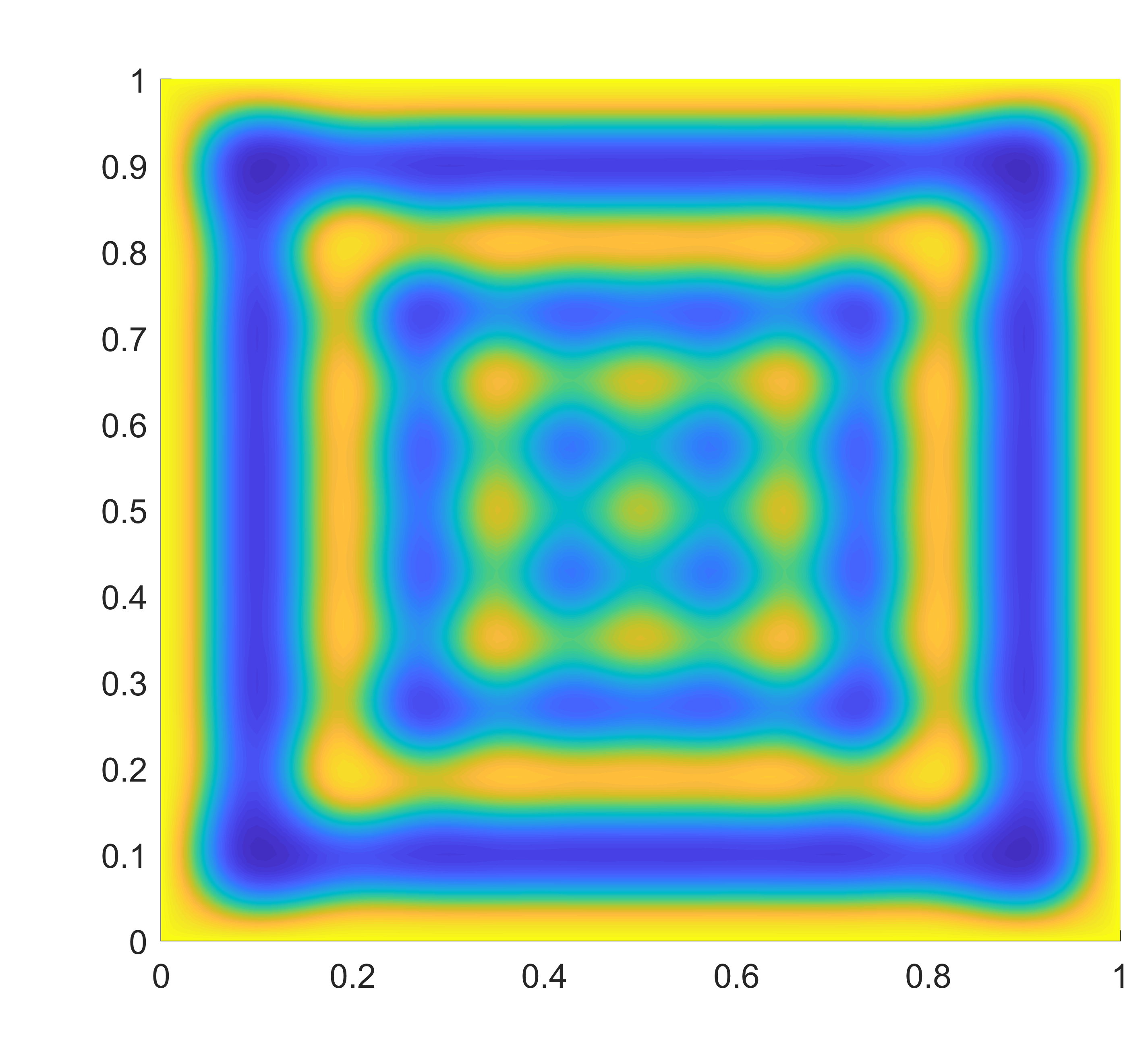}
	\includegraphics[width=0.32\textwidth]{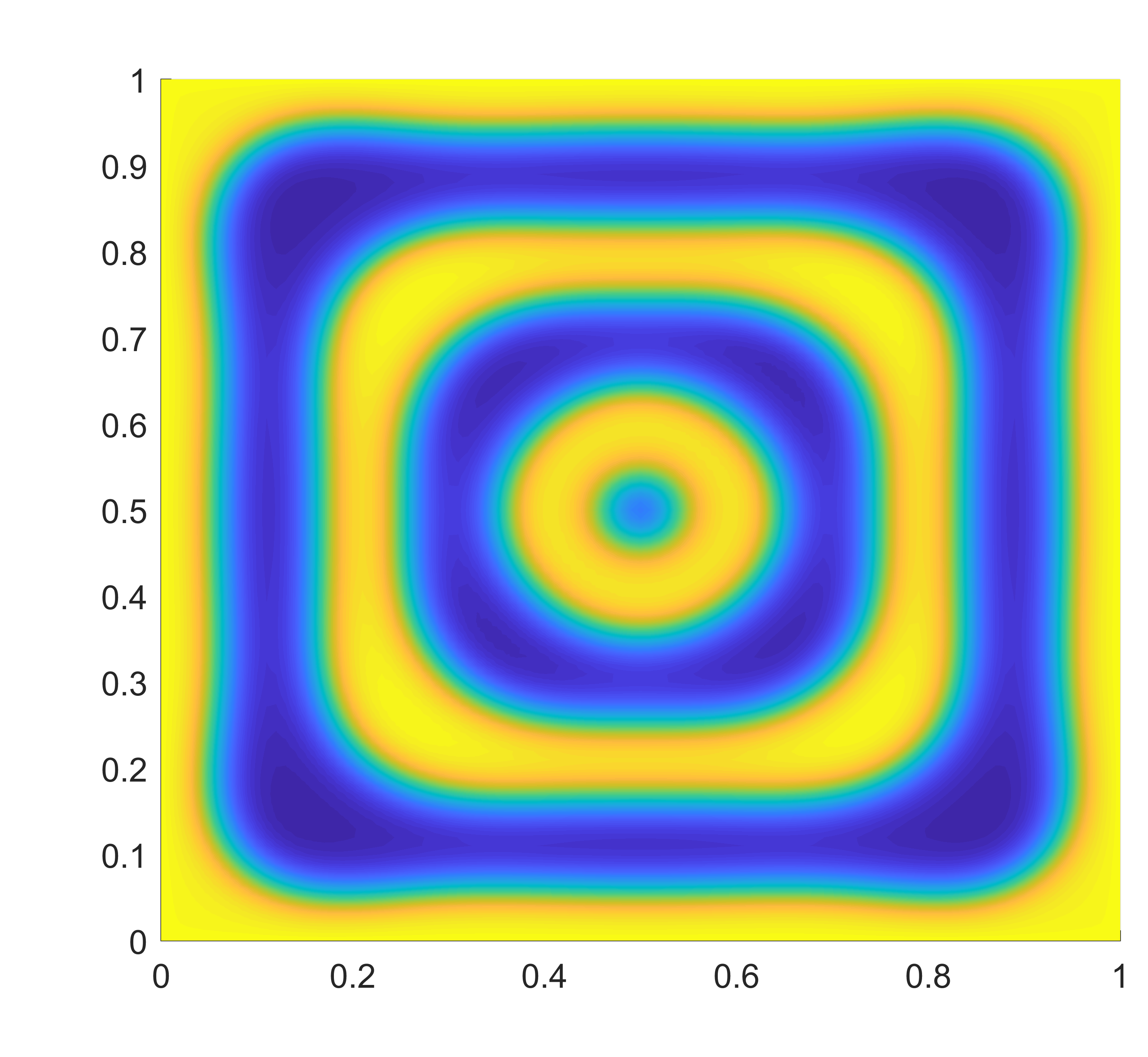}
	\\
	\includegraphics[width=0.32\textwidth]{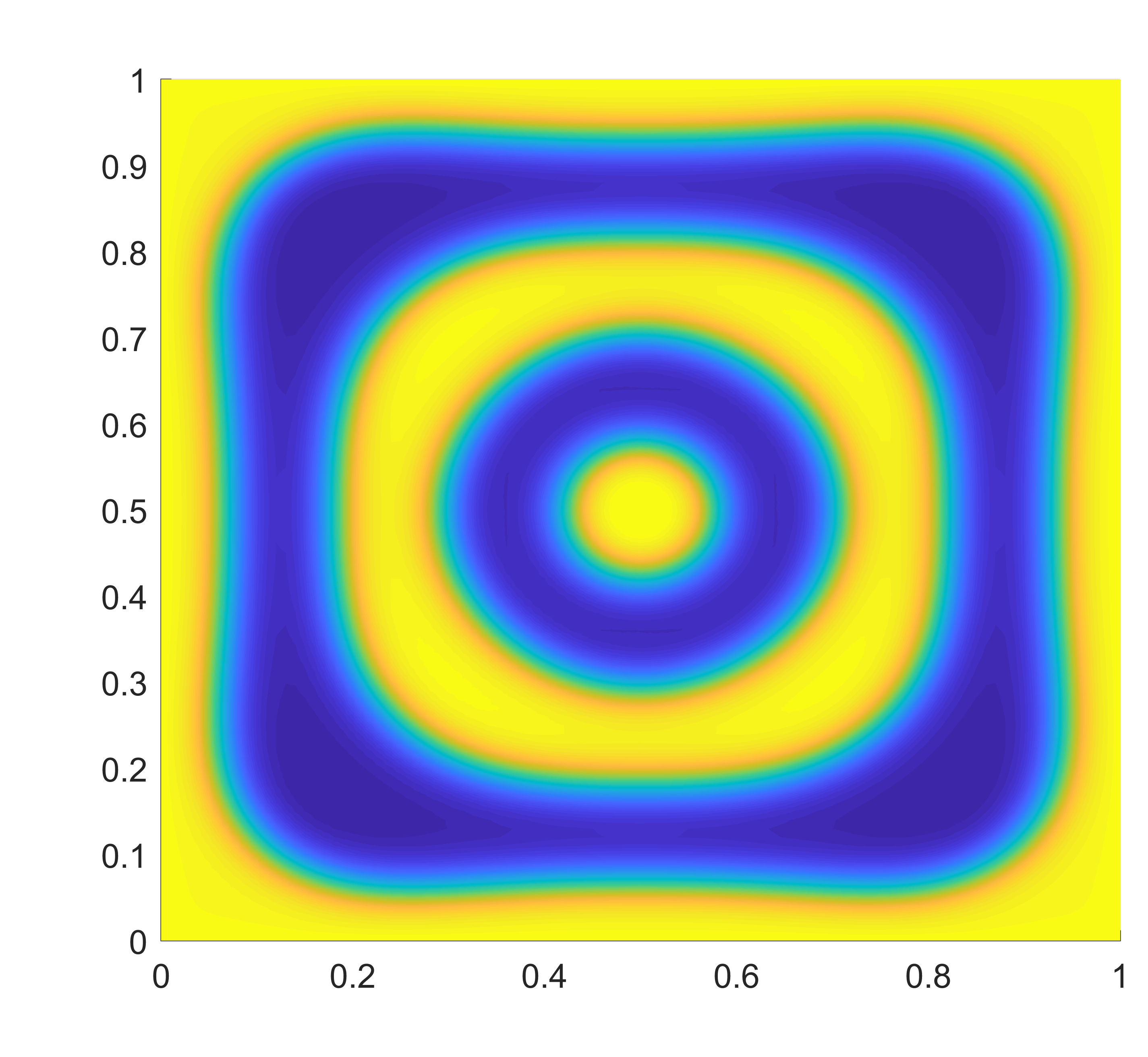}
	\includegraphics[width=0.32\textwidth]{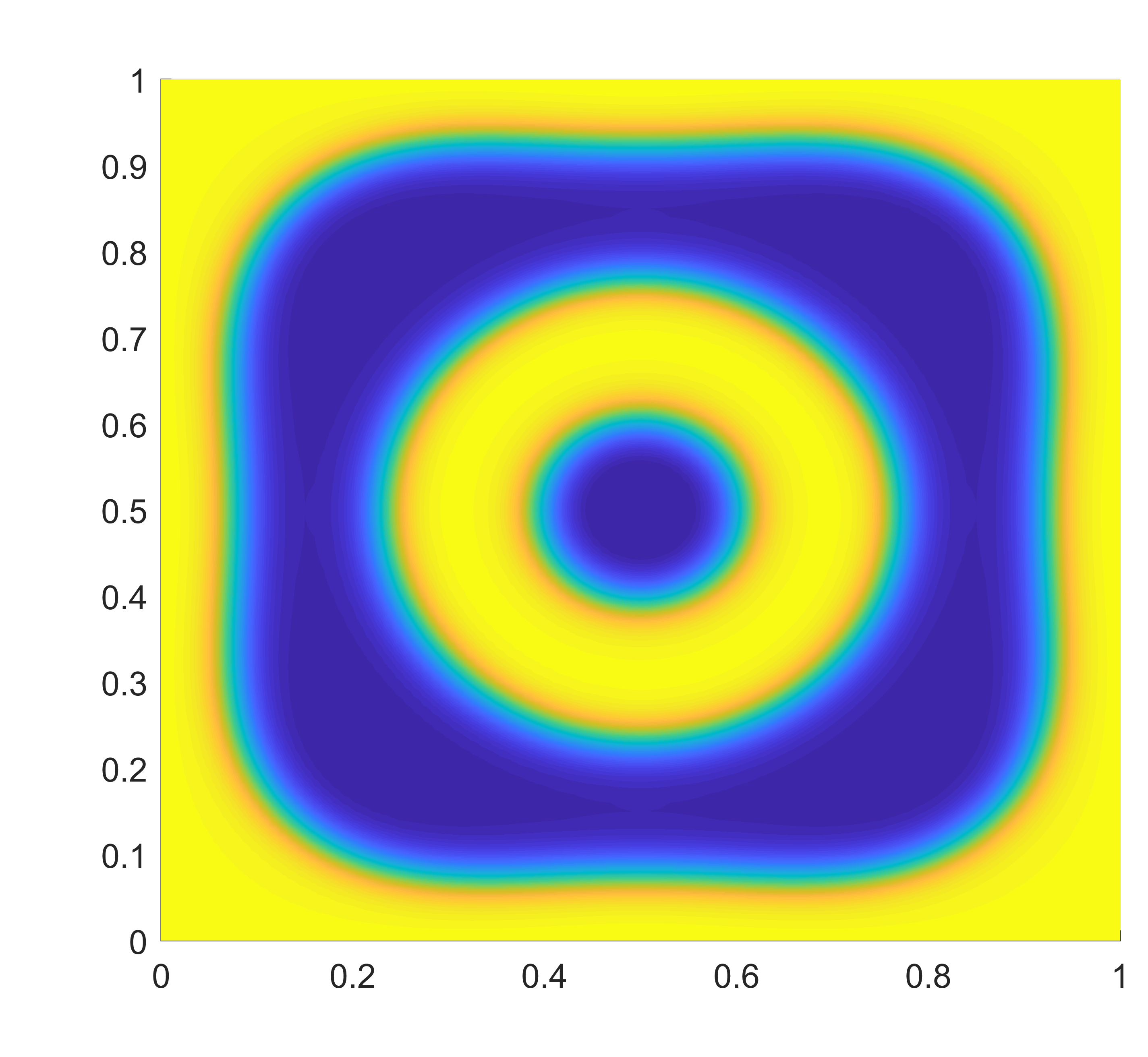}
	\includegraphics[width=0.32\textwidth]{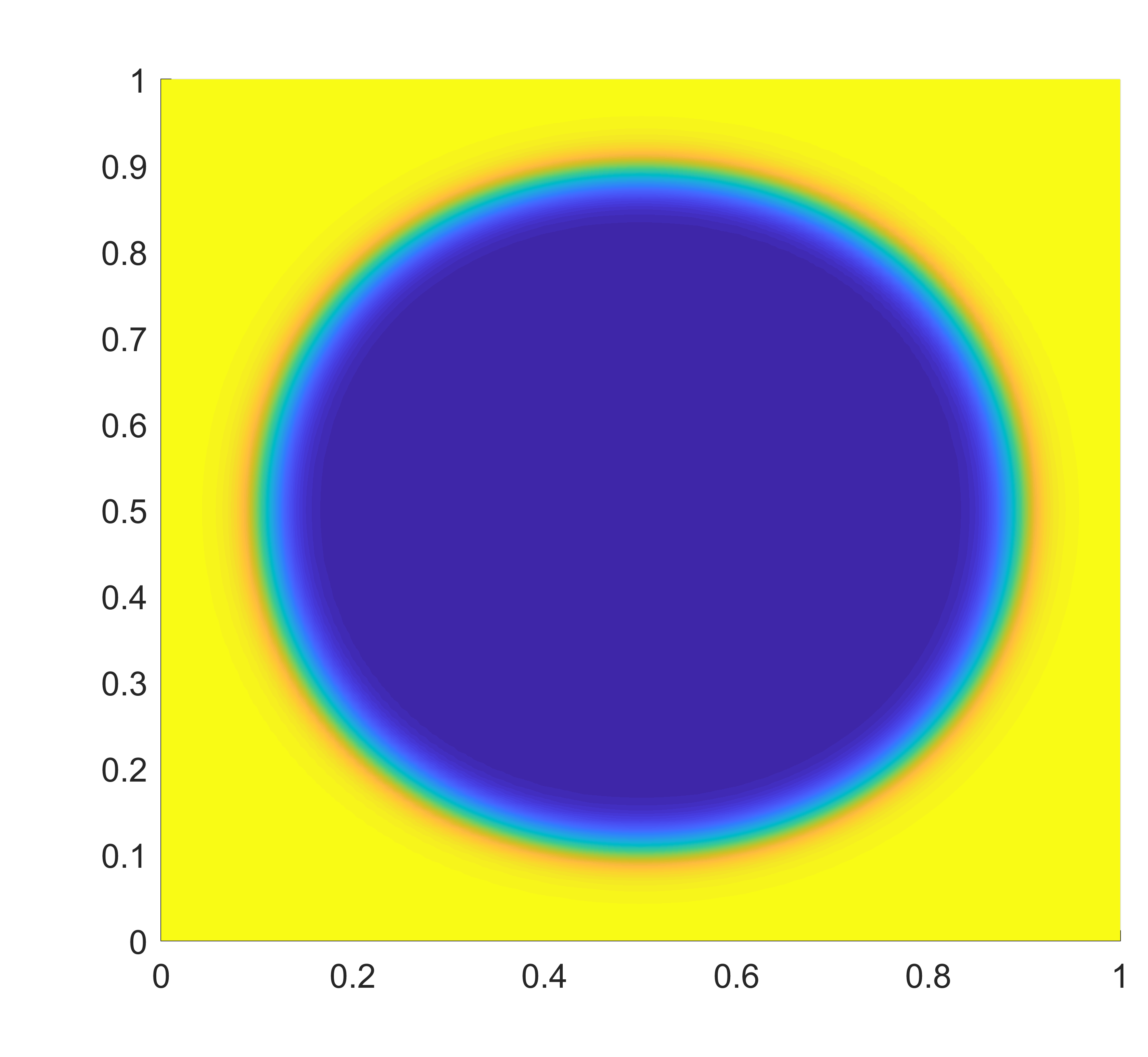}
	\\
	\includegraphics[width=50mm]{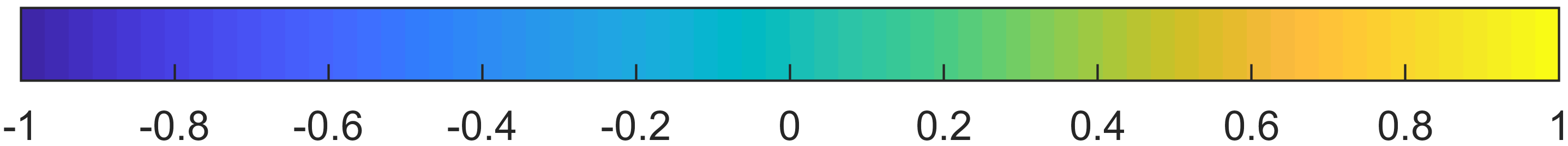}
	\caption{The initial datum $\phi_0$ is set to zero in $\Omega_h$ and set to one on $\Gamma_h$.}\label{fig1}
\end{figure}

\begin{figure}[h!]
	\centering
	\includegraphics[width=0.32\textwidth]{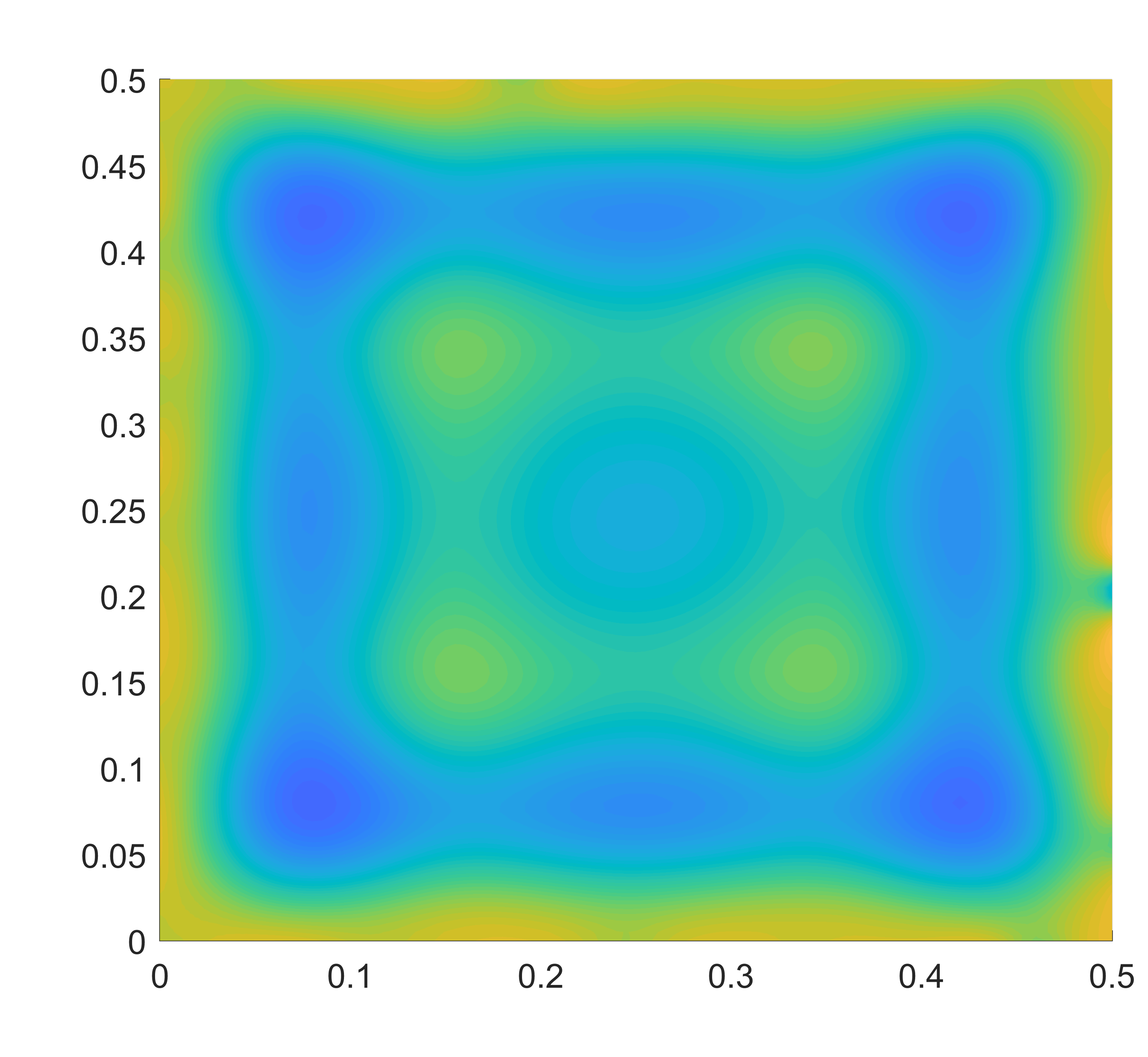}
	\includegraphics[width=0.32\textwidth]{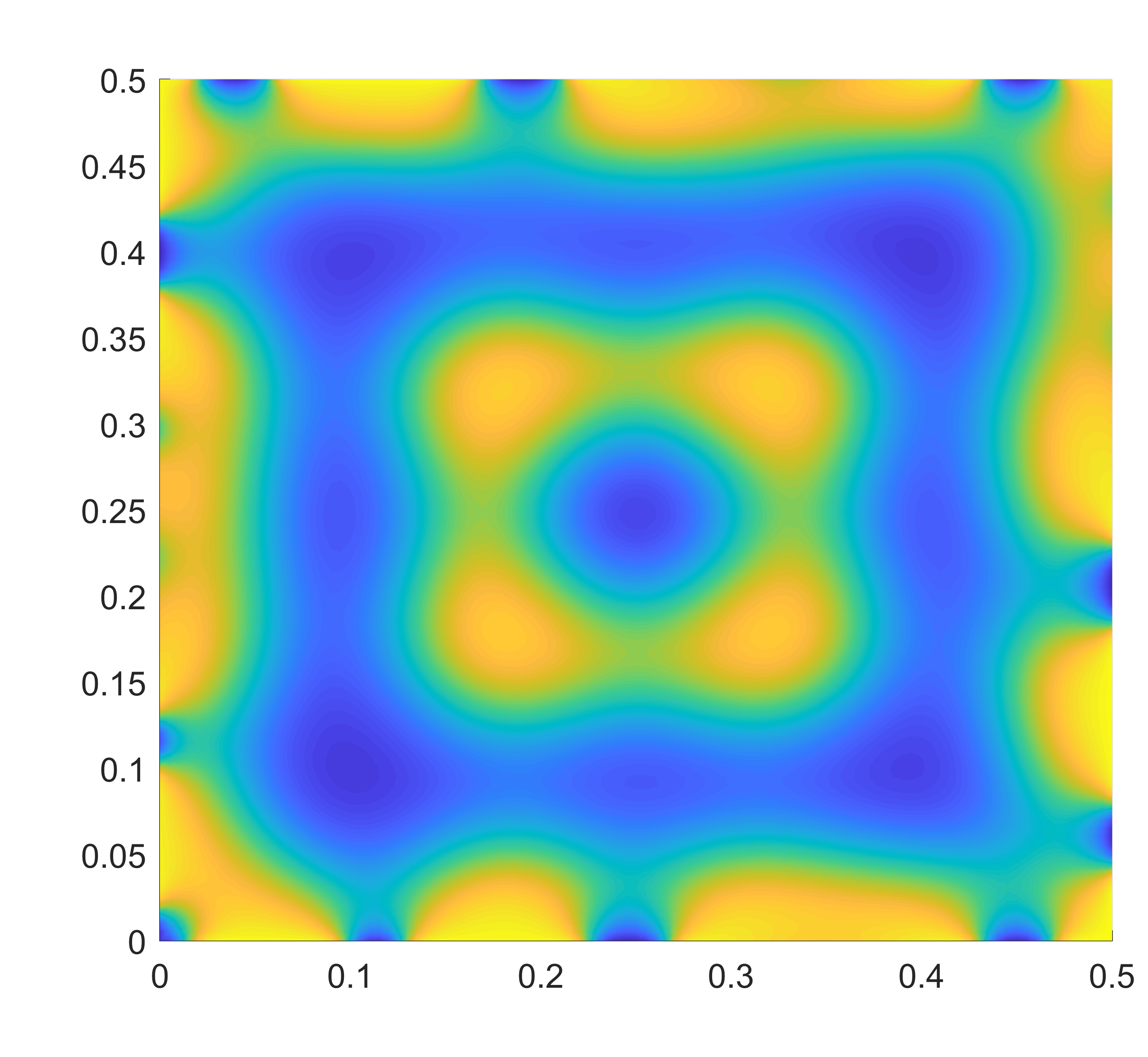}
	\includegraphics[width=0.32\textwidth]{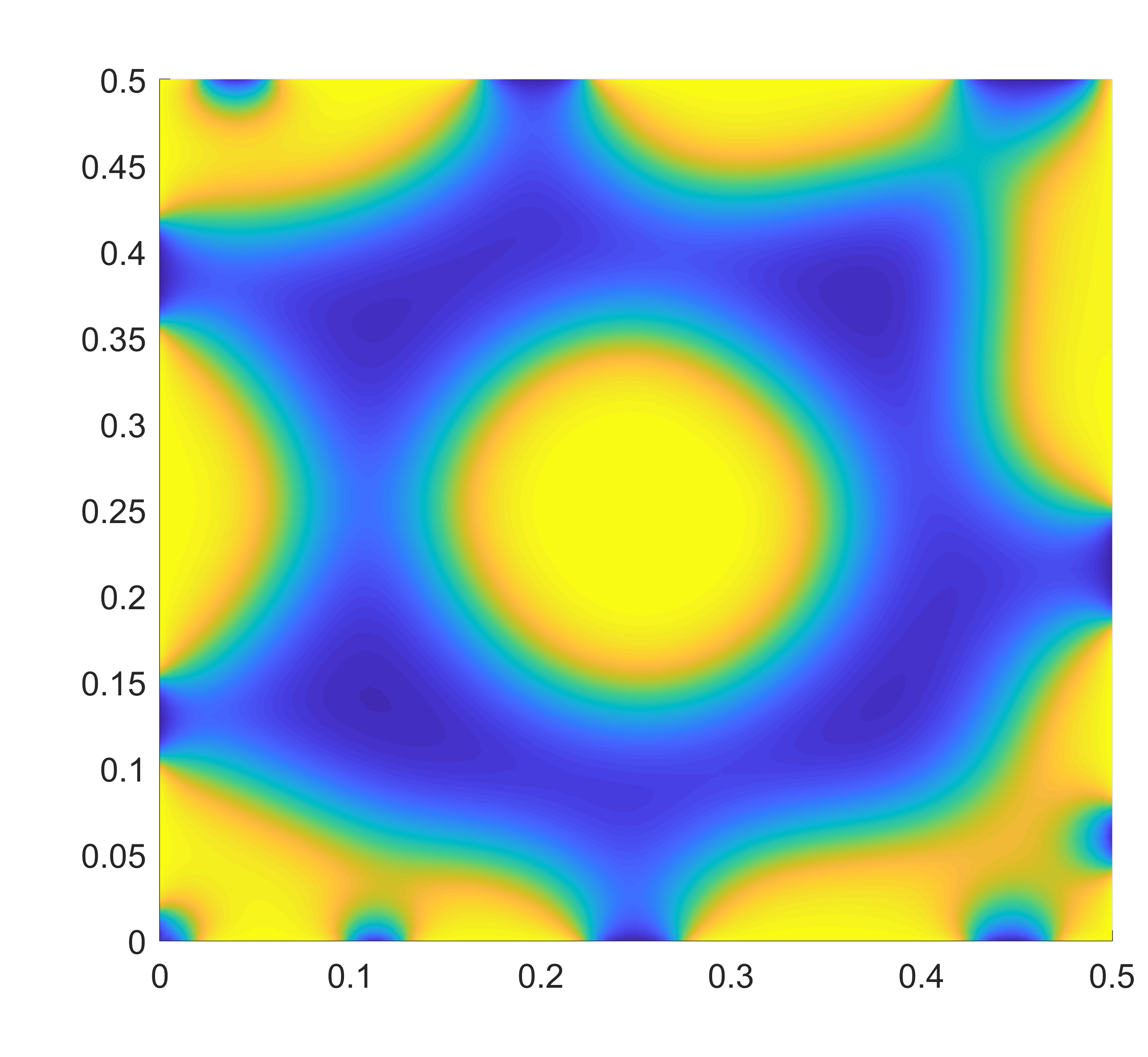}
	\\
	\includegraphics[width=0.32\textwidth]{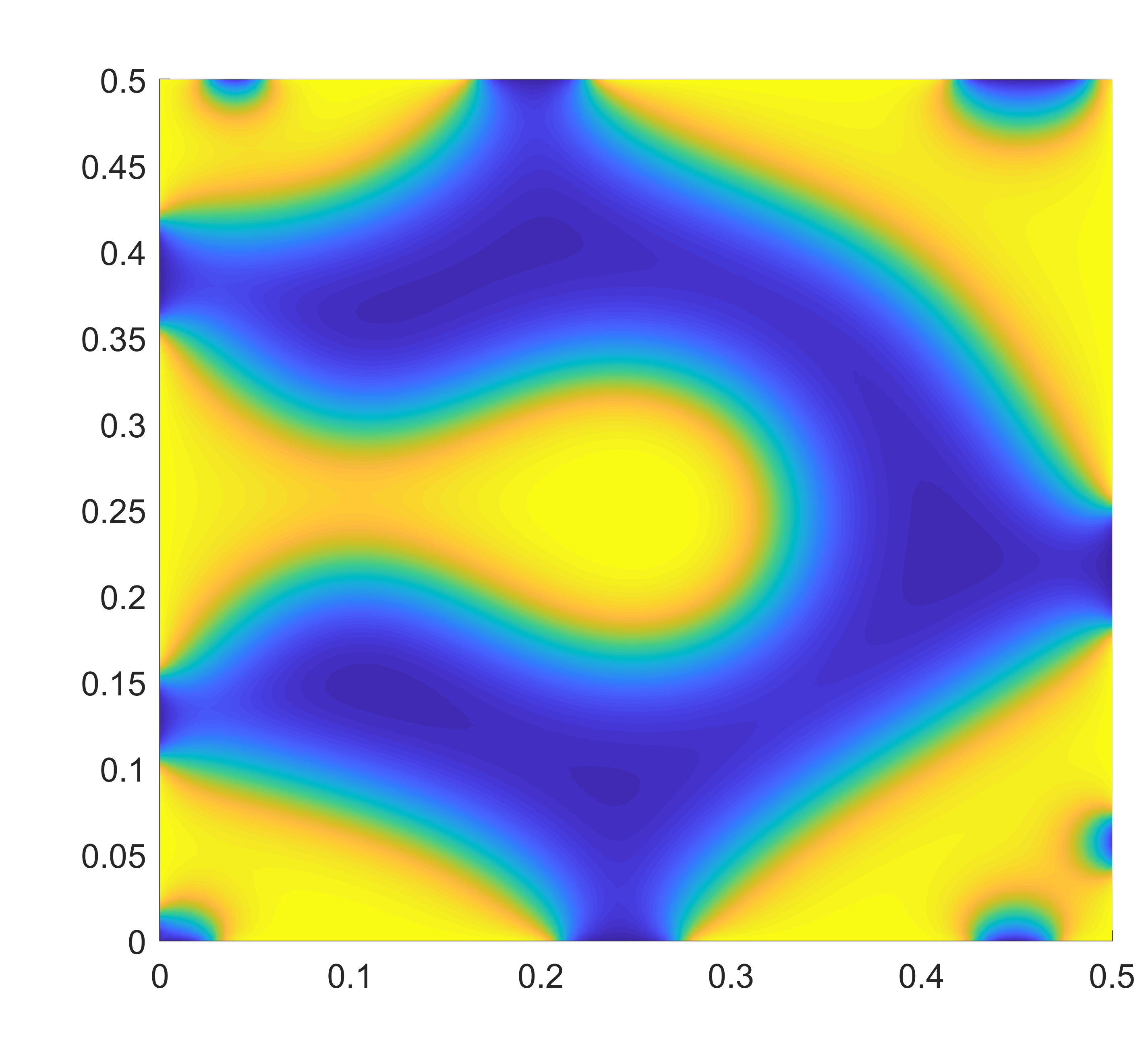}
	\includegraphics[width=0.32\textwidth]{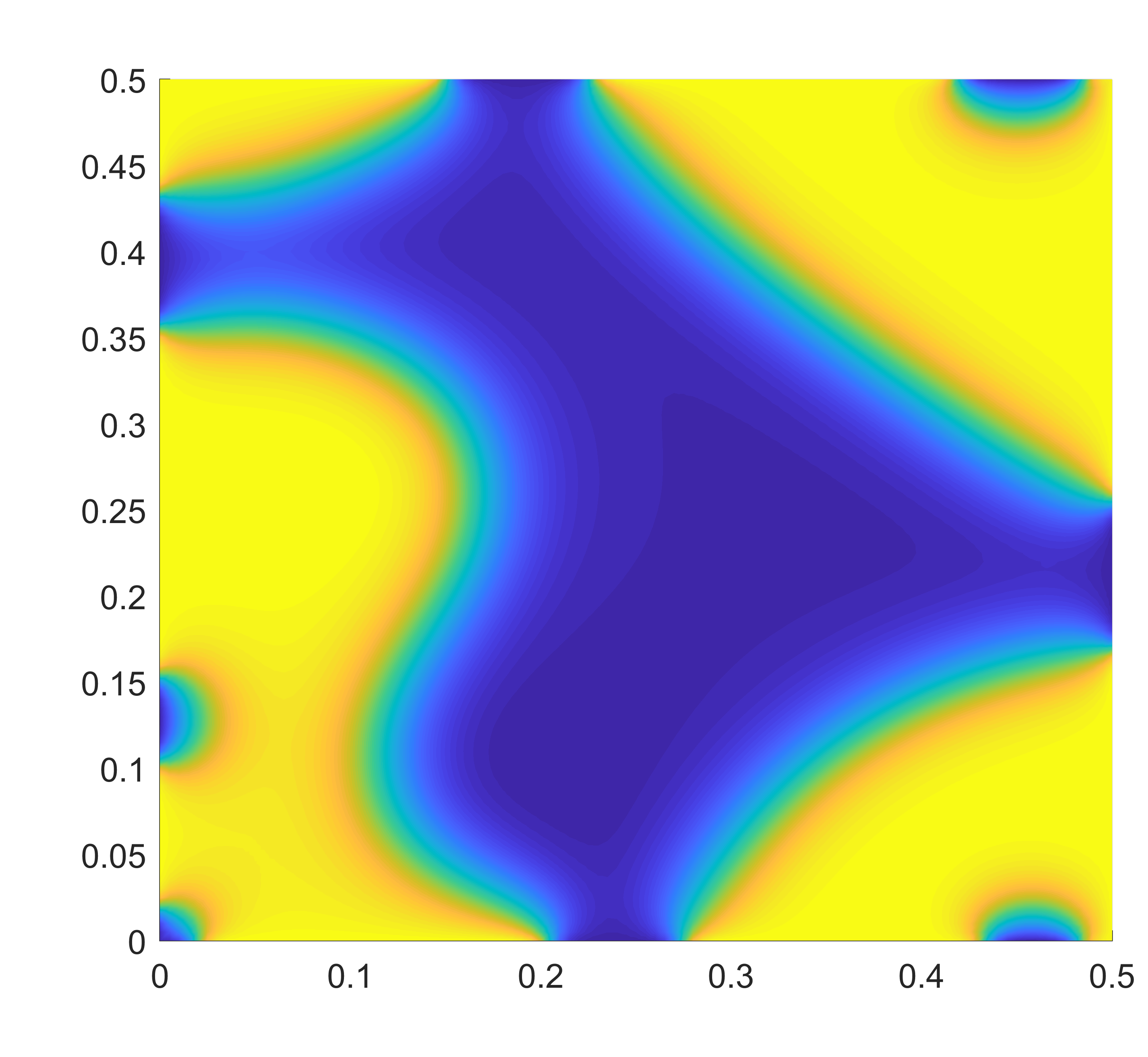}
	\includegraphics[width=0.32\textwidth]{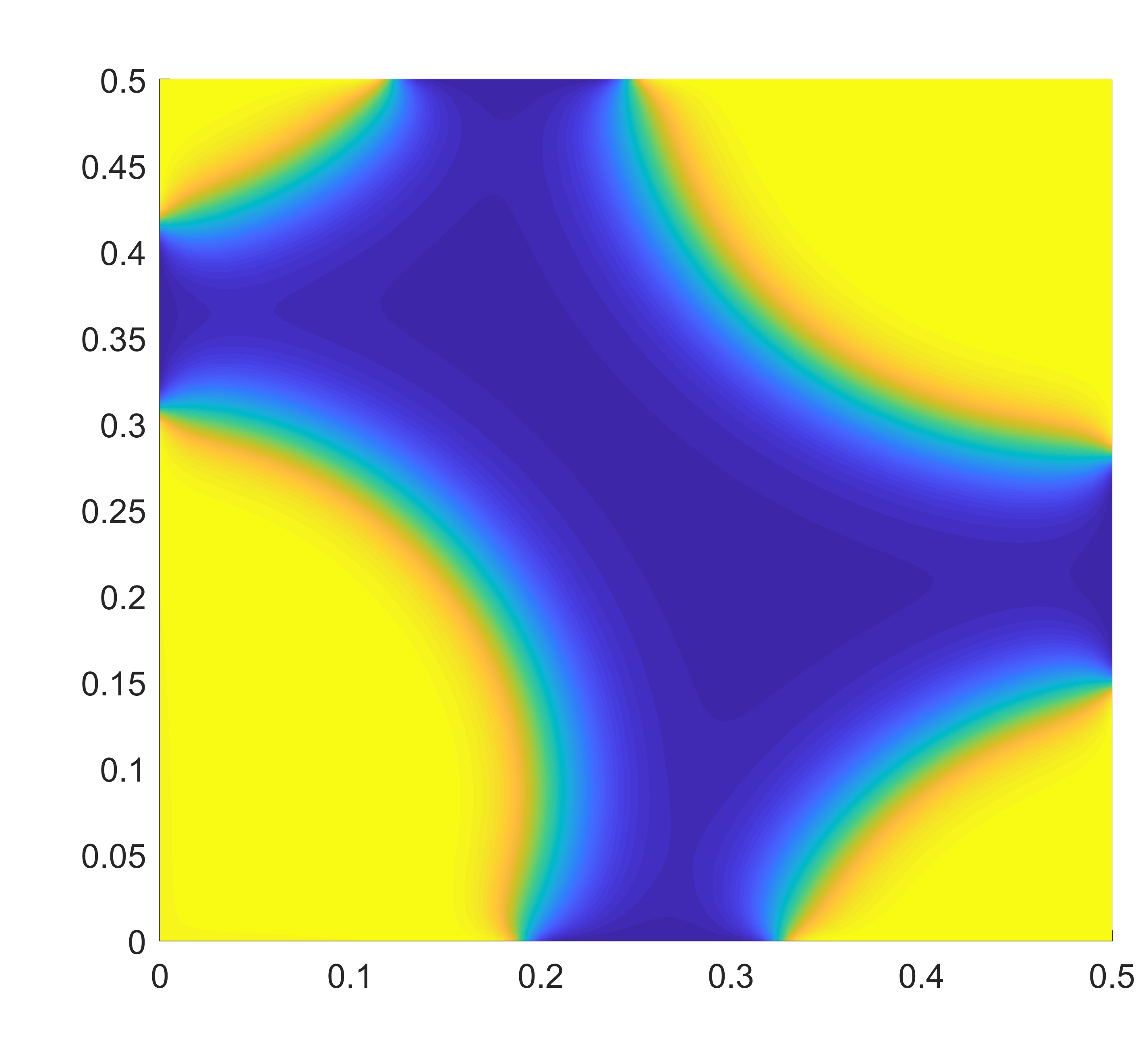}
	\\
	\includegraphics[width=50mm]{colorbar-hor}
	\caption{The initial datum $\phi_0$ attains random values between $-0.1$ and $0.1$ in $\Omega_h$ and random values between $0.4$ and $0.6$ on $\Gamma_h$. }\label{fig2}
\end{figure}

%
%

\FloatBarrier

\section*{Appendix}

For the reader's convenience, we state the generalized Poincar\'e inequality as presented by H.\,W. Alt \cite[p.\,242]{alt} because of its importance in the approach of Section 4.4:

\begin{lemma}[Generalized Poincar\'e inequality]
	\label{LEM:GPI}
	Let $\Omega\subset\mathbb{R}^n$ be open, bounded and connected with Lipschitz boundary $\del\Omega$. Moreover, let $1<p<\infty$ and let $\mathcal M\subset W^{1,p}(\Omega)$ be nonempty, closed and convex. Then the following items are equivalent for every $u_0$ in $\mathcal M$:
	\begin{enumerate}
		\itemi There exists a constant $C_0<\infty$ such that for all $\xi\in\R$,
		\begin{align*}
		u_0+\xi\in\mathcal M \Rightarrow |\xi|\le C_0.
		\end{align*}
		\itemii There exists a constant $C<\infty$ with 
		\begin{align*}
		\|u\|_{L^p(\Omega)} \le C\,\big( \|\grad u\|_{L^p(\Omega)} + 1 \big)\quad\text{for all}\; u\in\mathcal M.
		\end{align*}
	\end{enumerate}
\end{lemma}



\footnotesize
\bibliographystyle{plain}
\bibliography{chdbc}%


\end{document}